\documentclass[10pt, leqno]{article}
\usepackage{amsmath,amsfonts,amsthm,amscd,amssymb,verbatim,color} 
\usepackage[all]{xy}
\usepackage{fullpage}

\numberwithin{equation}{section}

\theoremstyle{plain}
\newtheorem{thm}{Theorem}[section]
\newtheorem{defn}[thm]{Definition}
\newtheorem{prop}[thm]{Proposition}
\newtheorem{lem}[thm]{Lemma}
\newtheorem{cor}[thm]{Corollary}

\theoremstyle{definition}
\newtheorem{rem}[thm]{Remark}

\newcommand{\beast}{\begin{eqnarray*}}
\newcommand{\east}{\end{eqnarray*}}

\newcommand{\N}{{\Bbb N}}
\newcommand{\Z}{{\Bbb Z}}
\newcommand{\Q}{{\Bbb Q}}

\newcommand{\Af}{{\Bbb A}}

\renewcommand{\Pr}{{\Bbb P}}

\newcommand{\A}{{\bold A}}

\renewcommand{\b}{{\bold b}}

\newcommand{\Spec}{{\mathrm{Spec}}\,}

\newcommand{\Spf}{{\mathrm{Spf}}\,}

\newcommand{\Zar}{{\mathrm{Zar}}}
\newcommand{\lra}{\longrightarrow}

\newcommand{\lla}{\longleftarrow}

\newcommand{\wt}[1]{\widetilde{#1}}

\newcommand{\ul}[1]{\underline{#1}}
\newcommand{\ol}[1]{\overline{#1}}
\newcommand{\os}{\overset}

\newcommand{\et}{{\mathrm{et}}}

\newcommand{\crys}{{\mathrm{crys}}}

\newcommand{\Cor}{{\mathrm{Cor}}}

\newcommand{\Fil}{{\mathrm{Fil}}}
\newcommand{\id}{{\mathrm{id}}}

\newcommand{\Inthom}{{\cal H}{\mathrm{om}}}

\newcommand{\pr}{{\mathrm{pr}}}
\newcommand{\dlog}{{\mathrm{dlog}}\,}

\newcommand{\Frac}{{\mathrm{Frac}}\,}

\newcommand{\cD}{{\cal D}}
\newcommand{\cE}{{\cal E}}
\newcommand{\cF}{{\cal F}}

\newcommand{\cH}{{\cal H}}
\newcommand{\cI}{{\cal I}}

\newcommand{\cO}{{\cal O}}
\newcommand{\cP}{{\cal P}}
\newcommand{\cQ}{{\cal Q}}

\newcommand{\cU}{{\cal U}}

\newcommand{\cX}{{\cal X}}
\newcommand{\cY}{{\cal Y}}

\newcommand{\fX}{{\mathfrak{X}}}
\newcommand{\fY}{{\mathfrak{Y}}}

\newcommand{\Bl}{{\mathrm{Bl}}}

\renewcommand{\tt}{{\bold t}}

\newcommand{\x}{{\bold x}}

\newcommand{\op}{{\mathrm{op}}}

\renewcommand{\wt}{\widetilde}

\newcommand{\e}{{\bold e}}

\newcommand{\MCorL}{\ul{\rm M}\Cor^{\Lambda}}
\newcommand{\MCorQ}{\ul{\rm M}\Cor^{\Q}}
\newcommand{\MCorLls}{\ul{\rm M}\Cor^{\Lambda, {\rm ls}}}
\newcommand{\MCorQls}{\ul{\rm M}\Cor^{\Q, {\rm ls}}}

\newcommand{\MCorZls}{\ul{\rm M}\Cor^{\Z, {\rm ls}}}
\newcommand{\Sm}{{\rm Sm}}
\newcommand{\Ab}{{\rm Ab}}
\newcommand{\Nis}{{\rm Nis}}
\renewcommand{\1}{{\bold 1}}
\renewcommand{\c}{{\bold c}}
\newcommand{\MDMeff}{\ul{\rm M}{\rm DM}^{\rm eff}}
\newcommand{\bMDWn}{\ul{\bf M}{\bf W}_n{\bf \Omega}^i}

\begin{document}
\title{Blow-up invariance for Hodge-Witt sheaves with modulus}
\author{Atsushi Shiho
\footnote{
Graduate School of Mathematical Sciences, 
University of Tokyo, 3-8-1 Komaba, Meguro-ku, Tokyo 153-8914, JAPAN. 
E-mail address: shiho@ms.u-tokyo.ac.jp \, 
Mathematics Subject Classification (2020): 14F30.}}
\date{}
\maketitle

\begin{abstract}
In this paper, we prove the blow-up invariance for Hodge-Witt sheaves with modulus, which is a generalization of a result of 
Koizumi for Witt sheaves and that of Kelly-Miyazaki and Koizumi for 
Hodge sheaves. As a consequence, we obtain the representability of 
Hodge-Witt sheaves with modulus in the category of motives with modulus under 
the assumption of resolution of singularities. 
\end{abstract}

\tableofcontents

\section{Introduction: modulus pairs}


The purpose of this article is to define the Hodge-Witt sheaves 
with modulus which are compatible with the Hodge sheaves with modulus 
by Kelly-Miyazaki \cite{KM23a}, \cite{KM23b} and Koizumi \cite{K23} 
and the Witt sheaves with modulus by Koizumi \cite{K23}, 
and prove the blow-up invariance of it. 

To explain our result precisely, first 
we recall several definitions concerning modulus pairs 
of Kahn-Miyazaki-Saito-Yamazaki \cite{KMSY21a}, \cite{KMSY21b}, \cite{KMSY22}, 
and recall the result of Koizumi \cite{K23} 
on the blow-up invariance of certain Nisnevich sheaves. 

Throughout the article, let $k$ be a perfect field. 
Let $\Sm$ be the category of smooth $k$-schemes. For 
$X, Y \in \Sm$, let $\Cor(X,Y)$ be the group of algebraic cycles 
on $X \times Y$ whose components are finite over $X$ and dominate some 
irreducible component of $X$. Let $\Cor$ be the category whose set of objects 
is the same as that of $\Sm$ and whose set of morphisms for 
$X, Y \in \Sm$ is given by $\Cor(X,Y)$. 
A morphism $f: X \lra Y$ in $\Sm$ is regarded as 
an element in $\Cor(X,Y)$ by identifying it with its graph. 
Also, a finite morphism $f: X \lra Y$ defines an 
element ${}^tf$ in $\Cor(Y,X)$ which is the transpose 
of the graph of $f$. 
For a presheaf $F:\Cor^{\op} \lra \Ab$ on 
$\Cor$ and a finite morphism $f: X \lra Y$, we denote the map 
$F(X) \lra F(Y)$ associated to ${}^tf$ by $f_*$ or ${\rm Tr}_{X/Y}$, 
and call it the pushforward map or the trace map associated to $f$. 

A presheaf $F:\Cor^{\op} \lra \Ab$ on $\Cor$ is called a Nisnevich sheaf 
if for any $X \in \Sm$, the presheaf 
$$ F_X: X^{\op}_{\Nis} \lra \Ab; \quad U \mapsto F(U) $$
is a sheaf, where $\Nis$ denotes the Nisnevich site. 

It is known that, for any $i \in \N$, 
the Hodge presheaf $X \mapsto \Omega^i_X$
on $\Sm$ is a Nisnevich sheaf $\Omega^i$ 
on $\Cor$ (\cite[Thm.~A.4.1]{KSY16}), which we call the 
Hodge sheaf. 
Also, when $k$ has characteristic $p>0$, 
it is known that, for any $n, i \in \N$, the Hodge-Witt presheaf 
$X \mapsto W_n\Omega^i_X$ on $\Sm$ is a Nisnevich sheaf 
$W_n\Omega^i$ 
on $\Cor$ (\cite[I Prop.1.14]{I79} and \cite[Thm.~B.2.2]{KSY16}), which 
we call the Hodge-Witt sheaf. 
We denote $\Omega^0, W_n\Omega^0$ also by 
$\cO, W_n\cO$ respectively, and we call $W_n\cO$ the Witt sheaf. 

For $\Lambda \in \{\Z,\Q\}$, a $\Lambda$-modulus pair 
is a pair $\fX = (X,D_X)$ consisting of an algebraic $k$-scheme and 
an effective $\Lambda$-Cartier divisor $D_X$ on $X$ such that 
$X^{\circ} := X \setminus |D_X|$ is smooth. 
The $\Lambda$-modulus pair $\ol{\square} := (\Pr^1_k, \{ \infty \})$ is called 
the cube. For $\Lambda$-modulus pairs $\fX, \fY$, let 
$\MCorL(\fX,\fY)$ be the subgroup of 
$\Cor(X^{\circ},Y^{\circ})$ consisting of cycles 
any of whose component $V$ satisfies the 
following conditions: 
\begin{itemize}
\item
The closure $\ol{V}$ of $V$ in $X \times Y$ is proper over $X$. 
\item The inequality of $\Lambda$-Cartier divisors 
$(\pr_1^*D_X)|_{\ol{V}^{\rm N}} \geq 
(\pr_2^*D_Y)|_{\ol{V}^{\rm N}}
$ holds, where $\pr_1: X \times Y \lra X, \pr_2: X \times Y \lra Y$ 
are projections and $\ol{V}^{\rm N}$ is the normalization of $\ol{V}$. 
\end{itemize}
Then, let $\MCorL$ be the category 
whose objects are $\Lambda$-modulus pairs and 
whose set of morphisms for $\Lambda$-modulus pairs 
$\fX, \fY$ is given by $\MCorL(\fX,\fY)$. 
For $\fX, \fY \in \MCorL(\fX,\fY)$, we denote the 
$\Lambda$-modulus pair 
$(X \times Y, \pr_1^*D_X + \pr_2^*D_Y)$ by 
$\fX \otimes \fY$.

An ambient morphism between $\Lambda$-modulus pairs 
$f: \fX \lra \fY$ is defined to be a morphism $f: X \lra Y$ of 
$k$-schemes with $f(X^{\circ}) \subseteq Y^{\circ}$, 
$D_X \geq f^*D_Y$. 
It is called minimal if it satisfies $D_X = f^*D_Y$ in addition. 
An ambient morphism $f$ is regarded as 
a morphism in $\MCorL$ by identifying it with its graph. 
Also, a finite minimal ambient morphism $f:\fX \lra \fY$ defines 
an element ${}^tf$ in $\MCorL(\fY,\fX)$ which is the transpose 
of the graph of $f$.

A presheaf $F:(\MCorL)^{\op} \lra \Ab$ on $\MCorL$ is called a Nisnevich sheaf 
if for any $\fX \in \MCorL$, the presheaf 
$$ F_{\fX}: X^{\op}_{\Nis} \lra \Ab; \quad U \mapsto F(U,D_X|_U) $$
is a sheaf.

We call a $\Lambda$-modulus pair $\fX = (X,D_X)$ log smooth 
if $X$ is smooth and $|D_X|$ is a simple normal crossing divisor on $X$, 
and let $\MCorLls$ be the full subcategory of 
$\MCorL$ consisting of log smooth $\Lambda$-modulus pairs. 
For $\fX \in \MCorLls$, we say that 
a closed subscheme $Z$ in $|D_X|$ has simple normal crossing with $|D_X|$ 
if, Zariski locally on $X$, there is a coordinate 
$x_1, \dots, x_n$ and $1 \leq a \leq b \leq c \leq n$ such that 
$|D_X| = \{ x_1x_2 \cdots x_b = 0\}, Z = \{x_1 = \cdots = x_a = 
x_{b+1} = \cdots x_c = 0\}$. In this situation, we can form the blow up 
$\varphi: \Bl_Z(\fX) \lra \fX$, where 
$\Bl_Z(\fX) = (\Bl_Z(X), D_X|_{\Bl_Z(X)})$. 

Then one of the main result of \cite{K23} is as follows: 

\begin{thm}[{Koizumi \cite[Thm.~1.4]{K23}}]\label{thm:koizumi1} 
Let $F$ be a Nisnevich sheaf on $\MCorQ$ satisfying the following three conditions$:$ \\
$(1)($Cohomological cube invariance$)$ \, 
For $\fX \in \MCorQls$, the morphism $R\Gamma(X,F_{\fX}) \lra 
R\Gamma(X \times \Pr^1, F_{\fX \otimes \ol{\square}})$ is a quasi-isomorphism. \\
$(2)($Affine vanishing property$)$ \, 
For any effective divisors $E_1, \dots, E_n$ on $\Af^1$, 
$\fX \in \MCorQls$ and $i>0$, 
$$ R^i\pr_{1,*}F_{\fX \otimes (\A^1,E_1) \otimes \cdots \otimes 
(\A^1,E_n)} = 0. $$ 
$(3)($Left continuity$)$ \, For any $\fX \in \MCorQls$, 
The canonical map 
$\varinjlim_{\epsilon \to 0}F(X,(1-\epsilon)D_X) \lra 
F(\fX)$ is an isomorphism. \\
Then, for any object $\fX \in \MCorQls$ and a closed subscheme $Z$ of 
$|D_X|$ which has simple normal crossing with $|D_X|$, the morphism 
$$ R\Gamma(X,F_{\fX}) \lra R\Gamma(\Bl_ZX, F_{\Bl_Z\fX}) $$
is an isomophism. 
\end{thm}

We call the property in the conclusion of the above theorem 
the blow-up invariance for $F$. 

Next we recall a way to construct a Nisnevich sheaf on 
$\MCorQ$ from a Nisnevich sheaf on $\Cor$, which is also due to 
Koizumi \cite{K23}. 
We say a henselian discrete valuation field $(L,v_L)$ isomorphic to 
${\rm Frac}\,\cO_{X,x}^{\rm h}$ for some 
$X \in \Sm$ and a codimension $1$ point $x$ of $X$ a geometric henselian 
discrete valuation field, and denote the set of geometric henselian 
discrete valuation fields by $\Phi$. 
For $L \in \Phi$, let $O_L$ be its valuation ring. 
For a Nisnevich sheaf $F$ on $\Cor$ and $L \in \Phi$, we can consider 
the value $F(L), F(O_L)$ at $L$ and at $O_L$ because $\Spec L, \Spec O_L$ 
can be written as the filtered projective limit of a system of 
objects in $\Sm$ (and it is easy to see the independence of the choice of the projective system). Also, 
for any finite extension $L \subseteq L'$ of objects in $\Phi$, 
we can define the trace morphism ${\rm Tr}_{L'/L}: F(L') \lra F(L)$ because 
the morphism $\Spec L' \lra \Spec L$ can be also written as the 
filtered projective limit of a system of 
finite morphisms $X' \lra X$ in $\Sm$ with 
$\Spec L \times_X X' = \Spec L'$. (See the beginning of Section 2 for the detail.) 

A ramification filtration $\Fil$ 
on a Nisnevich sheaf $F$ on $\Cor$ is 
a collection of increasing filtations $\{\Fil_r F(L)\}_{r \in \Q_{\geq 0}}$ on 
$F(L)$ for all $L \in \Phi$ such that the following two conditions are 
satisfied: 
\begin{itemize}
\item For all $L \in \Phi$, ${\rm Im}(F(O_L) \lra F(L)) \subseteq 
\Fil_0F(L)$. 
\item For any finite extension $L \subseteq L'$ of objects in $\Phi$ 
with ramification index $e$, ${\rm Tr}_{L'/L}(\Fil_r F(L')) \subseteq 
\Fil_{r/e}(F(L))$. 
\end{itemize}
Then, for $\fX \in \MCorQ$ we say that an element $a \in F(X^{\circ})$ 
is bounded by $D_X$ if, for any $L \in \Phi$ and any commutative 
diagram over $k$ 
\begin{equation}\label{eq:ramdiag}
\xymatrix{
\Spec L \ar[r]^-{\rho} \ar[d] & X^{\circ} \ar[d] \\
\Spec O_L \ar[r]^-{\tilde{\rho}} & X, 
}
\end{equation}
$\rho^*(a)$ belongs to $\Fil_{v_L(\tilde{\rho}^*D_X)}F(L)$. 
Then we define $F_{\Fil}(\cX)$ to be the set of sections in 
$F(X^{\circ})$ which are bounded by $D_X$. 
Then Koizumi proved that $F_{\Fil}$ is a Nisnevich sheaf on 
$\MCorQ$ (\cite[Lem.~2.5]{K23}). 

We recall two examples of the above construction which are also 
given in \cite{K23}. 

\begin{prop}\label{prop:koizumi}
$(1)$ {\rm \cite[Def.~4.1, Lem.~4.4]{K23}} \, 
Let $i \in \N$. Then the collection of increasing filtations 
$\{\Fil_r \Omega^i(L)\}_{r \in \Q_{\geq 0}}$ on $\Omega^i(L)$ for $L \in \Phi$ 
defined by 
\[ 
\Fil_r\Omega^i(L) = 
\left\{ 
\begin{aligned}
& \Omega^i(O_L) & (r=0) \\
& t^{-\lceil r \rceil + 1} \Omega^i(O_L)(\log t) & (r>0) 
\end{aligned}
\right. 
\] 
$($where $t$ denotes a uniformizer of $O_L)$ defines 
a ramification filtration on the Hodge sheaf $\Omega^i$. 
We denote the associated sheaf $(\Omega^i)_{\Fil}$ on $\MCorQ$ by 
$\ul{\rm M}\Omega^i$. \\
$(2)$ {\rm \cite[Def.~4.8, Lem.~4.11]{K23}} \, 
Let $n \in \N$. Then the collection of increasing filtations 
$\{\Fil_r W_n\cO(L)\}_{r \in \Q_{\geq 0}}$ on $W_n\cO(L)$ for $L \in \Phi$ 
defined by 
\[ 
\Fil_rW_n\cO(L) = 
\left\{ 
\begin{aligned}
& W_n(O_L) & (r=0) \\
& 
\{a \in W_n(L) \,|\, \ul{t}^{\lceil r \rceil -1}F^{n-1}(a) \in W_n(O_L)\} & 
(r>0)
\end{aligned}
\right. 
\] 
$($where $t$ is a uniformizer of $O_L$, $\ul{t}$ is the Teichm\"ller lift 
of $t$ in $W_n(O_L)$ and $F$ is the Frobenius operator on $W_n(L))$ 
defines a ramification filtration on the sheaf $W_n\cO$. 
We denote the associated sheaf $(W_n\cO)_{\Fil}$ on $\MCorQ$ by 
$\ul{\rm M}W_n\cO$. 
\end{prop}

Then the following theorem is known as 
the second main result of \cite{K23}. 

\begin{thm}[{Koizumi \cite[Cor.~4.7, Cor.~1.5]{K23}}]\label{thm:koizumi2}
$(1)$ \, Let $i \in \N$. Then the Nisnevich sheaf 
$\ul{\rm M}\Omega^i$ on $\MCorQ$ in Proposition \ref{prop:koizumi}$(1)$ 
satisfies the conditions $(1)$, $(2)$, $(3)$ of Theorem 
\ref{thm:koizumi1} $($hence satisfies the blow-up invariance$)$. \\
$(2)$ \, Assume $k$ has characteristic $p>0$ and let $n \in \N$. Then 
the Nisnevich sheaf 
$\ul{\rm M}W_n\cO$ on $\MCorQ$ in Proposition \ref{prop:koizumi}$(2)$ 
satisfies the conditions $(1)$, $(2)$, $(3)$ of Theorem 
\ref{thm:koizumi1} $($hence satisfies the blow-up invariance$)$. 
\end{thm}

We note that, when $k$ has characteristic zero, 
the blow-up invariance for $\ul{\rm M}\Omega^i$ in the case $\fX \in \MCorZls$ 
is proven first by Kelly-Miyazaki \cite{KM23a}, \cite{KM23b}. 

In Theorem \ref{thm:koizumi2}, (1) is the statament for Hodge sheaves, and 
(2) is the statement for Witt sheaves. So it would be natural to expect 
that the analogue of 
Proposition \ref{prop:koizumi} and Theorem \ref{thm:koizumi2} 
would be true also for Hodge-Witt sheaves. This is our main result. 

\begin{thm}\label{thm:main}
Assume $k$ has positive characteristic $p>0$ and let $i,n \in \N$. \\
$(1)$ \, 
There exists 
a ramification filtration on the Hodge-Witt sheaf $W_n\Omega^i$
which coincides with the ramification filtration on 
the Hodge sheaf $\Omega^i$ when $n=1$ and that on the Witt sheaf 
$W_n\cO$ when $i=0$. 
We denote the associated sheaf $(W_n\Omega^i)_{\Fil}$ on $\MCorQ$ by 
$\ul{\rm M}W_n\Omega^i$. \\
$(2)$ \, 
The Nisnevich sheaf 
$\ul{\rm M}W_n\Omega^i$ on $\MCorQ$ in $(1)$ 
satisfies the conditions $(1)$, $(2)$, $(3)$ of Theorem 
\ref{thm:koizumi1} $($hence satisfies the blow-up invariance$)$. 
\end{thm}

The main idea of the construction is to use the isomorphism 
\begin{equation}\label{eq:ir-intro}
W_n\Omega^i_X \cong R^iu_{X,*}\cO_{X/W_n} \cong 
\cH^i(\Omega^{\bullet}_{\cX/W_n})
\end{equation}
(where $X \in \Sm$, $\cX$ is a lift of $X$ to a smooth scheme over $W_n$
 and $u_X: (X/W_n)_{\crys}^{\sim} \lra X_{\Zar}^{\sim}$ is the projection from the crystalline topos to the Zariski topos) 
of Illusie-Raunaud \cite{IR83}. 
Using the filtration on 
$\Omega^{i}_{\cX/W_n}$ analogous to that of Kelly-Miyazaki \cite{KM23a}, \cite{KM23b} and Koizumi \cite{K23}, 
we define the ramification filtration on $W_n\Omega^i$. Precisely speaking, 
in order to define the filtration independently 
of the choice of the lift $\cX$, we define the filtration on the level of 
crystalline cohomology sheaves $R^iu_{X,*}\cO_{X/W_n}$. 
Also, we need certain renumbering of the filtration 
in order to match our filtration with that of Kelly-Miyazaki and Koizumi. 

As a consequence of our main result, we can prove the representability of 
Hodge-Witt cohomology with modulus in the category 
$\MDMeff$ of motives with modulus of Kahn-Miyazaki-Saito-Yamazaki (\cite{KMSY21a}, \cite{KMSY21b}, \cite{KMSY22}). 

\begin{cor}
Assume that $k$ has positive characteristic $p>0$ and that 
$k$ admits resolution of singularities in the sense of 
{\rm \cite[Def.~5.1]{K23}}. Then, for $i,n \in \N$, there exists an object 
$\bMDWn$ in $\MDMeff$ such that, for any $\fX = (X,D_X) \in \MCorZls$, 
we have an equivalence 
$$ {\rm map}_{\MDMeff}({\rm M}(\fX), \bMDWn) \simeq 
R\Gamma(X, \ul{\rm M}W_n\Omega^i_{\fX}), $$
where ${\rm M}(\fX)$ is the object in $\MDMeff$ associated to $\fX$. 
\end{cor}

\begin{proof}
The claim follows immediately from Theorem \ref{thm:main} 
and \cite[Lem.~5.3]{K23}. 
\end{proof}

In Section 2, we give a proof of Theorem \ref{thm:main}. 
In Appendix, we prove the compatibility of pushforward maps of 
Hodge-Witt sheaves and that of de Rham complex via the isomorphism 
\eqref{eq:ir-intro} of Illusie-Raynaud, which we use in Section 2. 

\section{Proofs}

First we discuss on projective systems of smooth $k$-schemes 
whose limit is the valuation ring of a given geometric henselian 
discrete valuation field. So let $L \in \Phi$, and denote by $O_L$, $k_L$ 
its valuation ring and its residue field respectively. 
Let $s_L$ be the closed point of $\Spec O_L$ and we regard the pair 
$(\Spec O_L, s_L)$ as a log scheme. 
Let $\cP_L$ be the set of all triples 
$(X,D,g)$ satisfying the following property (A): 

\begin{enumerate}
\item[(A)]\label{item:A}
$X = \Spec A_0$ is an affine smooth $k$-variety, $D$ is a smooth principal 
divisor in $X$, $g$ is a morphism $(\Spec O_L,s_L) \lra (X,D)$ of log schemes 
induced by an injective map $A_0 \hookrightarrow O_L$ of $k$-algebras. 
\end{enumerate}

We define the partial order $\succeq$ on $\cP_L$ in such a way that 
$(X',D',g') \succeq (X,D,g)$ if and only if there exists a 
morphism $\varphi:(X',D') \lra (X,D)$ such that $\varphi \circ g' = g$. 
Then, with this partial order, $\cP_L$ is a projective system. 
By the existence of an isomorphism 
$O_L \cong k_L[t_0]^{\rm h}$ (where $k_L[t_0]^{\rm h}$ denotes 
the henselization of the polynomial algebra $k_L[t_0]$ with respect to the ideal $(t_0)$), we see that $\varprojlim_{\cP_L} X = \Spec O_L$, 
$\varprojlim_{\cP_L} (X \setminus D) = \Spec L$. 

Next, when we choose a fixed uniformizer $t_0$ of $O_L$, 
we define $\cP'_L \subseteq \cP_L$ to be the sub projective system 
consisting of all the triples $(X,D,g)$ satisfying the property (A) and 
the following property (B): 

\begin{enumerate}
\item[(B)]\label{item:B} 
$t_0 \in A_0$, $D$ is defined by $t_0$, and there exists a smooth 
$k$-algebra $B_0$ and an etale map 
$\varphi_0: B_0[t_0] \lra A$ of $k$-algebras. 
\end{enumerate}

%

Noting the isomorphism $O_L \cong k_L[t_0]^{\rm h}$ above, 
we see that $\cP'_L$ is cofinal in $\cP_L$. 
In particular, we have $\varprojlim_{\cP'_L} X = \Spec O_L$, $\varprojlim_{\cP'_L} (X \setminus D) = \Spec L$. 

Next consider the situation 
that we are given a finite extension $L \subseteq L'$ of 
geometric henselian discrete valuation fields. 
Let $O_L, k_L, t_0$ be as before and denote by $O_{L'}, k_{L'}, t'_0$ 
the valuation ring of $L'$, the residue field of $L'$ 
and an uniformizer of $O_{L'}$ respectively. 
We define $\cP''_{L,L'}$ to be the projective system consisting of 
the tuples $(X,D,g,X',D',g',h_0)$ with $(X,D,g) \in \cP'_L$, 
$(X',D',g') \in \cP'_{L'}$ and a finite morphism of 
log schemes $h_0:(X',D') \lra (X,D)$ satisfying 
the following property (C): 

\begin{enumerate}
\item[(C)]\label{item:C}
The diagram 
\[ 
\xymatrix{
(\Spec O_{L'},s_{L'}) \ar[r]^-{g'} \ar[d] & 
(X',D') \ar[d]^{h_0} \\
(\Spec O_L,s_L) \ar[r]^-g & 
(X,D)}
\]
is cartesian. 
\end{enumerate}

We check that the projective system $\cP''_{L.L'}$ is useful: 

\begin{lem}\label{lem:cofinal}
If we denote by $\cP''_L$, $\cP''_{L'}$ the image of the map 
\begin{align*}
& \cP''_{L,L'} \lra \cP_L, \quad (X,D,g,X',D',g',h_0) \mapsto (X,D,g) \\
& \cP''_{L,L'} \lra \cP_{L'}, \quad (X,D,g,X',D',g',h_0) \mapsto (X',D',g')
\end{align*}
respectively, $\cP''_L$, $\cP''_{L'}$ are cofinal in 
$\cP_L, \cP_{L'}$ respectively. In particular, 
$\varprojlim_{\cP''_{L,L'}} X = \Spec O_L$, 
$\varprojlim_{\cP''_{L,L'}} (X \setminus D) = \Spec L$, 
$\varprojlim_{\cP''_{L,L'}} X' = \Spec O_{L'}$, 
$\varprojlim_{\cP''_{L,L'}} (X' \setminus D') = \Spec L'$. 
\end{lem}

\begin{proof}
Since $\Spec O_{L'}$ is finite flat over $\Spec O_L$, 
for a sufficiently large element $(X_1,D_1,g_1)$ in $\cP'_L$ with respect to 
$\succeq$, we can find a $k$-variety $X'_1$ and a finite flat morphism 
$h_1:X'_1 \lra X_1$ with $\Spec O_L \times_{X_1} X'_1 = \Spec O_{L'}$, and 
after replacing $(X_1,D_1,g_1)$ by a larger element in $\cP'_L$, 
we may assume that $X'_1$ is smooth, 
$D'_1 := h_1^{-1}(D_1)_{\rm red}$ is a smooth divisor of $X'_1$ and 
$(X'_1,D'_1)$ satisfies the properties (A), (B) 
(with $L, O_L, X, D, t_0$ replaced 
by $L', O_{L'}, X'_1, D'_1, t'_0$ respectively). 
Then $h_1$ defines a finite morphism $(X'_1,D'_1) \lra (X_1,D_1)$ 
of log schemes, and the tuple $(X_1,D_1,g_1,X'_1,D'_1,h_1)$ defines an element 
in $\cP''_{L,L'}$. Then, if we consider the projective system $\cP'''_{L,L'}$ 
consisting of tuples $(X,D,g,X'_1 \times_{X_1} X,D'_1 \times_{D_1} D,
h_1 \times \id: (X'_1 \times_{X_1} X, D'_1 \times_{D_1} D) \lra (X,D))$ 
with $(X,D,g) \in \cP'_L$ and $(X,D,g) \succeq (X_1,D_1,g_1)$, 
we see that $\cP'''_{L.L'} \subseteq \cP''_{L.L'}$ and that 
the image of $\cP'''_{L.L'}$ in $\cP_L, \cP_{L'}$ is cofinal in 
$\cP_L, \cP_{L'}$ respectively. So the proof is finished. 
\end{proof}

\begin{rem}\label{rem:indep-t0}
The projective system $\cP'_L$ depends on the choice of $t_0$, 
but if we define the projevtive system $(\cP'_L)'$ in the same way 
from another uniformizer $t'_0$, we can take the projective system 
$(\cP'_L)''$ consisting of all the triples 
$(X,D,g)$ satisfying the property (A) and 
the property (B) for both $t_0$ and $t'_0$, and 
$(\cP'_L)''$ is cofinal in both $\cP'_L$ and $(\cP'_L)'$. 
The same is true also for the projective system $\cP''_{L,L'}$. 
\end{rem}

\begin{rem}\label{rem:lift}
We give several remarks on lifts of schemes, algebras and maps between them over $k$ appeared above to those over $W_n$. \\
(1) \, If we are given an affine smooth $k$-variety $X = \Spec A_0$ and 
a smooth divisor $D$ of $X$, there exists a lift $(\cX, \cD)$ of 
$(X,D)$ such that $\cX = \Spec A$ is an affine scheme over $W_n$ and 
$\cD$ is a smooth divisor of $\cX$. \\
(2) \, If, in (1), $D$ is defined by an element $t_0$ in $A_0$, 
there exists a lift $t \in A$ of $t_0$ such that 
$\cD$ is defined by $t$. \\
(3) \, If in addition to (2) we are given a smooth $k$-algebra 
$B_0$ and an etale map $\varphi_0: B_0[t_0] \lra A_0$ of $k$-algebras, 
there exists a lift $B$ of $B_0$ to a smooth algebra over $W_n$ and 
an etale map $\varphi: B_0[t] \lra A$ over $W_n$ lifting $\varphi_0$. \\
(4) \, If $(X,D), (X', D')$ are as in (1), (2) 
and we are given a finite morphism 
$h_0:(X',D') \lra (X,D)$ of log schemes over $k$, 
there exists a lift $(\cX, \cD)$ of 
$(X,D)$ and a lift $(\cX', \cD')$ of $(X',D')$ as in (1), (2), endowed with 
a finite morphism $h:(\cX',\cD') \lra (\cX, \cD)$ over $W_n$ which lifts 
$h_0$. 
In particular, there exist some $u \in (A')^{\times}$ and $e \in \N_{>0}$ 
such that the image of $t$ by the map $A = \Gamma(\cX, \cO_{\cX}) \lra 
\Gamma(\cX', \cO_{\cX'}) = A'$ is written as $u(t')^e$. 
\end{rem}

Next we define meromorphic Hodge-Witt sheaves using 
log crystalline site. Let $X \in \Sm$ and $D = \bigcup_{a=1}^m D_a$ a simple normal crossing divisor of $X$ with each $D_a$ smooth. We regard the pair $(X,D)$ also as a log scheme. 
For $1 \leq a \leq m$, the smooth divisor $D_a$ defines a sub log structure 
$N_a$ of the log structure on $(X,D)$. 
For any object $T := ((U,M_U) \os{\iota}{\hookrightarrow} (T,M_T))$ 
of the log crystalline site $((X,D)/W_n)_{\crys}$ with respect to Zariski topology, 
$\iota$ induces the isomorphism 
$M_T/\cO_T^{\times} \cong M_U/\cO_U^{\times}$, and 
the latter contains a subsheaf of monoids $(N_{X,a}|_U)/\cO_U^{\times}$ which is 
locally isomorphic to $\N_U$. 
Thus we can take locally a generator of $(N_{X,a}|_U)/\cO_U^{\times}$ and send it to 
$M_T/\cO_T^{\times}$. Then the 
ideal sheaf $(\cI_{(X,D),a})_T$ of $\cO_T$ locally generated by the image of 
a lift to $M_T$ of the element of $M_T/\cO_T^{\times}$ in the previous sentence 
 by the structure morphism $M_T \lra \cO_T$ of the log structure $M_T$ is 
well-defined globally. Moreover,  $(\cI_{(X,D),a})_T$'s for 
$T$ in $((X,D)/W_n)_{\crys}$ forms an ideal 
$\cI_{(X,D),a}$ of the structure sheaf $\cO_{(X,D)/W_n}$ of 
$((X,D)/W_n)_{\crys}$. 
Then, for $\b := (b_1, \dots, b_m) \in \N^m$, we define 
$\cI_{(X,D)}^{\otimes \b}$ by 
$\cI_{(X,D)}^{\otimes \b} := \bigotimes_{a=1}^m 
\cI_{(X,D),a}^{\otimes b_a}$, and define 
$\cI_{(X,D)}^{\otimes (-\b)}$ by 
$\cI_{(X,D)}^{\otimes (-\b)} := \Inthom_{\cO_{(X,D)/W_n}}
(\cI_{(X,D)}^{\otimes \b}, \cO_{(X,D)/W_n})$. 
Then we call the sheaves 
$R^iu_{X,*}\cI_{(X,D)}^{\otimes (-\b)} 
\, (\b \in \N^m, i \in \N)$, where 
$u_{X}: ((X,D)/W_n)_{\crys}^{\sim} \allowbreak \lra X_{\Zar}^{\sim}$ 
is the projection 
from the log crystalline topos to the Zariski topos, the meromorphic 
Hodge-Witt sheaves. It is a module over $R^0u_{X,*}\cO_{(X,D)/W_n} \cong 
W_n\cO_X$, where the isomorphism follows from 
the isomorphism of Illusie-Raunaud 
\cite[III (1.5.2)]{IR83} because $R^0u_{X,*}\cO_{(X,D)/W_n}$ is unchanged 
when we change $D$ to the empty divisor. 

We give a concrete description of $R^iu_{X,*}\cI_{(X,D)}^{\otimes (-\b)}$ 
when $(X,D)$ is lifable to a pair $(\cX, \cD)$ of a smooth scheme $\cX$ over 
$W_n$ and a relative simple normal crossing divisor $\cD = \bigcup_{a=1}^m 
\cD_a$ on it. In this case, the module with integrable log connection 
\begin{equation}\label{eq:a-intconn}
\nabla: \cO_{\cX}(\textstyle\sum_{a=1}^m b_a\cD_a) \lra 
\cO_{\cX}(\textstyle\sum_{a=1}^m b_a\cD_a) 
\otimes_{\cO_{\cX}} \Omega^1_{\cX}(\log \cD)
\end{equation}
defines a crystal on $((X,D)/W_n)_{\crys}$, which we denote by $\cF_{\b}$. 
Then, if we denote by 
$Q: ((X,D)/W_n)_{\rm Rcrys}^{\sim} \lra ((X,D)/W_n)_{\crys}^{\sim}$
 the canonical morphism from the restricted crystalline topos to 
the usual crystalline topos, we can prove in the same way as 
\cite[p.156]{NS08} (see also \cite[(5.3)]{T99}) that 
$Q^*\cI_{(X,D),a}$ is canonically isomorphic to $Q^*\cF_{-\e_a}$, 
where $\e_a \in \Z^m$ is the $a$-th canonical basis. 
Thus, for $\b \in \N^m$, we have the canonical isomorphism 
$Q^*\cI_{(X,D)}^{\otimes(-\b)} \cong Q^*\cF_{\b}$. 
So we have isomorphisms 
\begin{align*}
R^iu_{X,*}\cI_{(X,D)}^{\otimes(-\b)} & 
\cong R^i(u_X \circ Q)_*Q^*\cI_{(X,D)}^{\otimes(-\b)} 
\cong R^i(u_X \circ Q)_*Q^*\cF_{\b} 
\cong R^iu_{X,*}\cF_{\b} \\ 
& \cong 
\cH^i(\Omega^{\bullet}_{\cX/W_n}(\log \cD)(\textstyle\sum_{a=1}^m b_a\cD_a)), 
\end{align*}
where the first and the third isomorphisms follows from 
\cite[Cor.~1.6.2]{NS08} 
and we denoted the log de Rham complex associated to the 
integrable log connection \eqref{eq:a-intconn} by 
$\Omega^{\bullet}_{\cX/W_n}(\log \cD)(\textstyle\sum_{a=1}^m b_a\cD_a)$. 
If $\cX$ admits a local coordinate $t_1, \dots, t_d$ for some $d \geq m$ such that 
$\cD_a$ is the zero locus of $t_a$ for $1 \leq a \leq m$, the complex 
$\Omega^{\bullet}_{\cX/W_n}(\log \cD)(\textstyle\sum_{a=1}^m b_a\cD_a)$ 
is also written as 
$\tt^{-\b} \Omega^{\bullet}_{\cX/W_n}(\log )$, where 
$\tt^{\b} = \prod_{a=1}^m t_a^{-b_a}$ and 
$\Omega^{j}_{\cX/W_n}(\log) \, (j \in \N)$ means the sheaf of 
log differential forms with respect to $\dlog t_1, \dots, \dlog t_m$. 

Note that we can work also with log crystalline site 
$((X,D)/W_n)_{\crys, \et}$ with respect to etale topology. 
So we have the sheaf $\cI_{(X,D),\et}^{\otimes(-\b)}$ on it and 
the etale version of meromorphic Hodge-Witt sheaves 
$R^iu_{X,\et,*}\cI_{(X,D),\et}^{\otimes(-\b)}$, where 
$u_{X,\et}: ((X,D)/W_n)_{\crys, \et}^{\sim} \allowbreak \lra X_{\et}$ is 
the etale analogue of $u_X$. It is a module over 
$R^0u_{X,\et, *}\cO_{(X,D)/W_n, \et} \cong 
W_n\cO_{X,\et}$. When 
 $(X,D)$ is lifable to a pair $(\cX, \cD)$ as above, 
$R^iu_{X,\et,*}\cI_{(X,D),\et}^{\otimes(-\b)}$ is isomorphic to 
$\cH^i(\Omega^{\bullet}_{\cX/W_n}(\log \cD)(\textstyle\sum_{a=1}^m b_a\cD_a)_{\et})$, where $\Omega^{\bullet}_{\cX/W_n}(\log \cD)(\textstyle\sum_{a=1}^m b_a\cD_a)_{\et}$ is the etale version of the complex 
$\Omega^{\bullet}_{\cX/W_n}(\log \cD)(\textstyle\sum_{a=1}^m b_a\cD_a)$. 
When $\cX$ admits a local coordinate $t_1, \dots, t_d$ as above, 
this complex is also written as 
$\tt^{-\b} \Omega^{\bullet}_{\cX/W_n}(\log )_{\et}$, which is the etale version 
of $\tt^{-\b} \Omega^{\bullet}_{\cX/W_n}(\log )$. 

\begin{prop}\label{prop:inj}
Let $X \in \Sm$ and $D = \bigcup_{a=1}^m D_a$ a simple normal crossing divisor of $X$ with each $D_a$ smooth. Also, let $\b \in \N^m$, $i \in \N$. Then, \\
$(1)$ \, $R^iu_{X,*}\cI_{(X,D)}^{\otimes(-\b)}$ is a quasi-coherent 
$W_n\cO_X$-module and $R^iu_{X,\et, *}\cI_{(X,D),\et}^{\otimes(-\b)}$ 
is a quasi-coherent $W_n\cO_{X,\et}$-module. Also, 
we have the isomorphism 
$R^iu_{X,*}\cI_{(X,D)}^{\otimes(-\b)} \os{\cong}{\lra}
\epsilon_* R^iu_{X,\et, *}\cI_{(X,D),\et}^{\otimes(-\b)}$, 
where $\epsilon: X_{\et}^{\sim} \lra X_{\Zar}^{\sim}$ is the canonical map 
from the etale topos to the Zariski topos. 

In particular, when $X$ is affine, we have the isomorphism 
$$ 
H^i(((X,D)/W_n)_{\crys}, \cI_{(X,D)}^{\otimes(-\b)}) 
\cong \Gamma(X, R^iu_{X,*}\cI_{(X,D)}^{\otimes(-\b)}). $$

\noindent
$(2)$ \, Assume $\b - \e_a \in \N^m$. Then the canonical map 
\begin{equation}\label{eq:a-crys-inj}
R^iu_{X,*}\cI_{(X,D)}^{\otimes(-\b+\e_a)}
\lra R^iu_{X,*}\cI_{(X,D)}^{\otimes(-\b)} 
\end{equation}
is injective. \\
$(3)$ \, If $b_a$ is not divisible by $p$, the map 
\eqref{eq:a-crys-inj} is an isomorphism. \\
$(4)$ \, Put $U := X \setminus D$ and denote the canonical 
open immersion $U \hookrightarrow X$ by $j$. 
Then the canonical map 
$$ 
R^iu_{X,*}\cI_{(X,D)}^{\otimes(-\b)} 
\lra j_*R^iu_{U,*}\cO_{U/W_n} 
$$
$($where $\cO_{U/W_n}$ is the structure crystal on $(U/W_n)_{\crys}$ and 
$u_U: (U/W_n)_{\crys}^{\sim} \lra U_{\Zar}^{\sim}$ is the projection from 
the crystalline topos to the Zariski topos$)$ 
is injective. 
\end{prop}

\begin{proof}
(1) \, First we prove the quasi-coherence of 
$R^iu_{X,*}\cI_{(X,D)}^{\otimes(-\b)}$ as 
$W_n\cO_X$-module. Since the assertion is local, 
we may assume that 
$(X,D)$ is lifable to a pair $(\cX, \cD)$ of a smooth scheme $\cX$ over 
$W_n$ and a relative simple normal crossing divisor $\cD = \bigcup_{a=1}^m 
\cD_a$ on it, and that 
$\cX$ admits a local coordinate $t_1, \dots, t_d$ for some $d \geq m$ with 
$\cD_a$ the zero locus of $t_a$ for $1 \leq a \leq m$. 
In this case, $R^iu_{X,*}\cI_{(X,D)}^{\otimes(-\b)} = 
\cH^i(\tt^{-\b} \Omega^{\bullet}_{\cX/W_n}(\log ))$. 
We prove the assertion by induction on the pair 
$(\dim X, |\b|)$. The case $|\b|=0$ follows from the equality 
$R^iu_{X,*}\cO_{(X,D)/W_n} = W_n\Omega^i_{(X,D)}$ 
(this is the definition given in \cite{HK94}): 
The quasi-coherence of it follows from the existence of weight filtration on 
$W_n\Omega^i_{(X,D)}$ (\cite[1.4]{M93}) and the quasi-coherence in non-logarithmic case. 
So we assume that $|\b| > 0$. 
We may assume that $b_1 > 0$ by change of indices, 
where $\b= (b_1, \dots, b_m)$. 
By considering the long exact sequence associated to the exact sequence 
\[ 
\xymatrix{
0 \ar[r] & 
\tt^{-\b+\e_1} \Omega^{\bullet}_{\cX/W_n}(\log ) 
\ar[r] & 
\tt^{-\b} \Omega^{\bullet}_{\cX/W_n}(\log )
\ar[r] & 
\dfrac{\tt^{-\b} \Omega^{\bullet}_{\cX/W_n}(\log )}{\tt^{-\b+\e_1} \Omega^{\bullet}_{\cX/W_n}(\log ) } \ar[r] & 0
}
\]
and the induction hypothesis, we see that it suffices to prove the quasi-coherence of 
$
\cH^i\left(\dfrac{\tt^{-\b} \Omega^{\bullet}_{\cX/W_n}(\log )}{\tt^{-\b+\e_1} \Omega^{\bullet}_{\cX/W_n}(\log )} \right)$ for any $i$. 

Let $\cY := \Af^d_{W_n}$, let $\cX \lra \cY$ 
be the map defined by the coordinate $t_1, \dots, t_d$ and let 
$\cX' \lra \cY'$ be the zero locus of $t_1$ of the previous map. 
Then we have the decomposition 
\begin{equation}\label{eq:decomp-omega0}
\tt^{-\b}\Omega^{j}_{\cX/W_n}(\log) = 
\tt^{-\b}(\cO_{\cX} \otimes_{\cO_{\cY'}} \Omega^j_{\cY'/W_n}(\log')) \oplus 
\tt^{-\b}(\cO_{\cX} \otimes_{\cO_{\cY'}} \Omega^{j-1}_{\cY'/W_n}(\log')) \wedge \dlog t_1 
\end{equation}
for any $j \in \N$, 
where $\Omega^*_{\cY'/W_n}(\log') \, (*=j-1,j)$ means the sheaf of 
log differential forms with respect to $\dlog t_2, \dots, \dlog t_m$. 
So we have the decomposition 
\begin{align}
\frac{\tt^{-\b}\Omega^{j}_{\cX/W_n}(\log)}{\tt^{-\b+\e_1}\Omega^{j}_{\cX/W_n}(
\log)} 
& = \tt^{-\b}(\cO_{\cX'} \otimes_{\cO_{\cY'}} \Omega^j_{\cY'/W_n}(\log')) \oplus 
\tt^{-\b}(\cO_{\cX'} \otimes_{\cO_{\cY'}} \Omega^{j-1}_{\cY'/W_n}(\log')) \wedge \dlog t_1 \label{eq:decomp-omega} \\
& = \tt^{-\b} (\Omega^j_{\cX'/W_n}(\log')) \oplus 
\tt^{-\b} (\Omega^{j-1}_{\cX'/W_n}(\log')) \wedge \dlog t_1 
\nonumber \\ 
& = 
t_1^{-b_1} (\tt^{-\b'} \Omega^j_{\cX'/W_n}(\log')) \oplus 
t_1^{-b_1} (\tt^{-\b'} \Omega^{j-1}_{\cX'/W_n}(\log')) \wedge \dlog t_1 
\nonumber 
\end{align}
for any $j \in \N$, where $\b' = (0,b_2, \dots, b_m)$. The differential of an element 
$t_1^{-b_1} \omega = t_1^{-b_1}\alpha + t_1^{-b_1}\beta \wedge \dlog t_1 \, 
(\alpha \in \tt^{-\b'}\Omega^j_{\cX'/W_n}(\log'), \beta \in 
\tt^{-\b'}\Omega^{j-1}_{\cX'/W_n}(\log'))$ 
in \eqref{eq:decomp-omega} is calculated as follows: 
\begin{align}
d(t_1^{-b_1}\omega) & = d(t_1^{-b_1}\alpha + t_1^{-b_1}\beta \wedge \dlog t_1) 
\label{eq:calcdiff} \\
& = -b_1t^{-b_1}\dlog t_1 \wedge \alpha + t_1^{-b_1}d\alpha 
+ t_1^{-b_1}d\beta \wedge \dlog t_1 \nonumber \\ & 
=  t_1^{-b_1}d\alpha + t^{-b_1}((-1)^{j+1}b_1\alpha + d \beta) \wedge 
\dlog t_1. 
\nonumber 
\end{align}
Hence there exists an exact sequence 
\begin{equation}\label{eq:exactseq}
\xymatrix{
0 \ar[r] & 
\tt^{-\b'} \Omega^{\bullet}_{\cX'/W_n}(\log')[-1] 
\ar[r]^{\phi} & 
\dfrac{\tt^{-\b}\Omega^{\bullet}_{\cX/W_n}(\log)}{\tt^{-\b+\e_1}
\Omega^{\bullet}_{\cX/W_n}(\log)} 
\ar[r]^{\psi} & 
\tt^{-\b'} \Omega^{\bullet}_{\cX'/W_n}(\log')
\ar[r] & 0,
}
\end{equation}
where $\phi$ is defined by 
$\beta \mapsto t_1^{-b_1}\beta \wedge \dlog t_1$ and 
$\psi$ is defined by 
$ t_1^{-b_1}\alpha + t_1^{-b_1}\beta \wedge \dlog t_1 
\mapsto \alpha$ via the identification \eqref{eq:decomp-omega}. 
By considering the long exact sequence associated to the exact sequence 
\eqref{eq:exactseq}, 
we are reduced to proving the quasi-coherence of 
$\cH^i(\tt^{-\b'} \Omega^{\bullet}_{\cX'/W_n}(\log'))$ for any $i$, 
and it follows from the induction hypothesis. So we are done. 
The quasi-coherence of 
$R^iu_{X,\et, *}\cI_{(X,D),\et}^{\otimes(-\b)}$ as 
$W_n\cO_{X,\et}$-module can be proven in the same way. 

Next we prove the isomorphism 
$R^iu_{X,*}\cI_{(X,D)}^{\otimes(-\b)} \os{\cong}{\lra}
\epsilon_* R^iu_{X,\et, *}\cI_{(X,D),\et}^{\otimes(-\b)}$. 
Let $\epsilon_{\crys}: ((X,D)/W_n)_{\crys, \et} \lra 
((X,D)/W_n)_{\crys}$ be the canonical map of topoi. Then we have 
the canonical map $\epsilon_{\crys}^*I_{(X,D)}^{\otimes(-\b)} \lra 
I_{(X,D),\et}^{\otimes(-\b)}$ and so the map 
\begin{equation}\label{eq:etzar0}
Ru_{X,*}I_{(X,D)}^{\otimes(-\b)} \lra 
Ru_*R\epsilon_{\crys,*}I_{(X,D),\et}^{\otimes(-\b)} 
= R\epsilon_*Ru_{X,\et,*}I_{(X,D),\et}^{\otimes(-\b)}. 
\end{equation}
We prove this is an isomorphism. To do so, 
we may work locally and so 
we may assume the existence of $(\cX, \cD)$ and $t_1, \dots, t_d$ as before. 
Then the map \eqref{eq:etzar0} is nothing but the map 
$$ \tt^{-\b}\Omega^{\bullet}_{\cX/W_n}(\log) 
\lra R\epsilon_* \tt^{-\b}\Omega^{\bullet}_{\cX/W_n}(\log)_{\et}, $$
which is an isomorphism because 
$\tt^{-\b}\Omega^j_{\cX/W_n}(\log)_{\et}$ is 
the quasi-coherent sheaf associated to 
$\tt^{-\b}\Omega^{\bullet}_{\cX/W_n}(\log)$. Hence 
the map \eqref{eq:etzar0} is an isomorphism. 
Then, since $R^ju_{X,\et,*}I_{(X,D),\et}^{\otimes(-\b)}$ is quasi-coherent, 
$R^h\epsilon_*R^ju_{X,\et,*}I_{(X,D),\et}^{\otimes(-\b)} \allowbreak = 0$ for 
$j,h \in \N$ with $h>0$ 
and so the map \eqref{eq:etzar0} induces the isomorphism  
\begin{equation}\label{eq:etzar}
R^iu_{X,*}I_{(X,D)}^{\otimes(-\b)} \os{\cong}{\lra}
\epsilon_*R^iu_{X,\et,*}I_{(X,D),\et}^{\otimes(-\b)}, 
\end{equation} 
as required. 

The last statement of (1) follows from the vanishing 
$H^h(X,R^iu_{X,*}\cI_{(X,D)}^{\otimes(-\b)}) = 0 \, (h > 0)$, 
which is a consequence of the quasi-coherence of 
$R^iu_{X,*}\cI_{(X,D)}^{\otimes(-\b)}$. 

(2) \, Without loss of generality, we may assume that $a=1$. 
Since the assertion is local, 
we may assume that there exists an etale morphism 
$X \lra \Af^d_k = \Spec k[t_{0,1}, \dots, t_{0,d}]$ 
for some $d \geq m$ such that each $D_{a'} \, (1 \leq a' \leq m)$ 
is the inverse image of the coordinate hyperplane 
$\{t_{0,a'}=0\}$ of $\Af^d_k$. 
Let $E$ be $\bigcup_{a'=1}^m \{t_{0,a'} = 0\}$.
By the isomorphism \eqref{eq:etzar} 
and the isomorphism \eqref{eq:etzar} with $\b$ replaced by $\b - \e_1$, 
it suffices to prove the injectivity of the map 
\begin{equation}\label{eq:a-crys-inj-et}
R^iu_{X,\et,*}\cI_{(X,D),\et}^{\otimes(-\b+\e_1)}
\lra R^iu_{X,\et,*}\cI_{(X,D),\et}^{\otimes(-\b)}. 
\end{equation}
Because the map \eqref{eq:a-crys-inj-et} is the restriction of 
the map \eqref{eq:a-crys-inj-et} with $(X,D)$ replaced by 
$(\Af^d_k, E)$, we are reduced to proving the injevitity of the latter map, 
namely, we may assume that $(X,D) = (\Af^d_k, E)$.

Let $\cX$ be $\Af^d_{W_n} = \Spec W_n[t_1, \dots, t_d]$ and let 
$\cD$ be $\bigcup_{a'=1}^m \{t_{a'} = 0\}$. Then the map 
\eqref{eq:a-crys-inj-et} is rewritten as 
\begin{equation}\label{eq:a-a-inj-sh}
\cH^i(\tt^{-\b+\e_1}\Omega^{\bullet}_{\cX/W_n}(\dlog)_{\et}) 
\lra H^i(t^{-\b}\Omega^{\bullet}_{\cX/W_n}(\dlog)_{\et}). 
\end{equation}
To prove the injectivity of the map 
\eqref{eq:a-a-inj-sh}, it suffices to prove  
the surjectivity of the map 
\begin{equation}\label{eq:a-a-surj}
\cH^i(\tt^{-\b+\e_1}\Omega^{\bullet}_{\cX/W_n}(\dlog)_{\et}) 
\lra \cH^i\left(\frac{\tt^{-\b}\Omega^{\bullet}_{\cX/W_n}(\dlog)_{\et}}{\tt^{-\b+\e_1}\Omega^{\bullet}_{\cX/W_n}(\dlog)_{\et}}\right)
\end{equation}
for any $i \in \N$. 
%

Let $\cX' = \Af^{d-1}_{W_n} = \Spec W_n[t_2, \dots, t_d]$. 
Then we have the decomposition 
\begin{equation}\label{eq:decomp-omega0-2}
\tt^{-\b}\Omega^i_{\cX/W_n}(\log)_{\et} = 
\tt^{-\b}(\cO_{\cX,\et} \otimes_{\cO_{\cX',\et}} 
\Omega^i_{\cX'/W_n}(\log')_{\et})
\oplus 
\tt^{-\b}(\cO_{\cX,\et} \otimes_{\cO_{\cX',\et}} \Omega^{i-1}_{\cX'/W_n}(\log')_{\et}) \wedge \dlog t_1 
\end{equation}
which is the etale version of 
\eqref{eq:decomp-omega0} in the setting here, and 
it induces the decomposition 
\begin{equation}\label{eq:decomp-omega-2}
\frac{\tt^{-\b}\Omega^{i}_{\cX/W_n}(\log)_{\et}}{\tt^{-\b+\e_1}\Omega^{i}_{\cX/W_n}(\log)_{\et}} = 
t_1^{-b_1} (\tt^{-\b'} \Omega^i_{\cX'/W_n}(\log')_{\et}) \oplus 
t_1^{-b_1} (\tt^{-\b'} \Omega^{i-1}_{\cX'/W_n}(\log')_{\et}) \wedge \dlog t_1, 
\end{equation}
which is the etale version of \eqref{eq:decomp-omega} in the setting here. 
The differential of an element 
$t_1^{-b_1} \omega = t_1^{-b_1}\alpha + t_1^{-b_1}\beta \wedge \dlog t_1 \, 
(\alpha \in \tt^{-\b'}\Omega^i_{\cX'/W_n}(\log')_{\et}, \beta \in 
\tt^{-\b'}\Omega^{i-1}_{\cX/W_n}(\log')_{\et})$ 
in \eqref{eq:decomp-omega-2} is calculated as in 
\eqref{eq:calcdiff} (with $j$ replaced by $i$). Hence 
it is a cochain if and only if 
\[ d\alpha = 0, \quad (-1)^{i+1}a\alpha + d \beta = 0  \] 
in $\Omega^{i+1}_{\cX'/W_n}(\log')_{\et}$. Now assume that the element 
$t_1^{-b_1} \omega$ is a cochain. 
Then, if we regard $\alpha, \beta$ as an element of 
$\cO_{\cX,\et} \otimes_{\cO_{\cX',\et}} \Omega^{i}_{\cX'/W_n}(\log')_{\et}$, 
$\cO_{\cX,\et} \otimes_{\cO_{\cX',\et}} \Omega^{i-1}_{\cX'/W_n}(\log')_{\et}$
 respectively via the map 
$\Omega^{j}_{\cX'/W_n}(\log')_{\et} \lra 
\cO_{\cX,\et} \otimes_{\cO_{\cX',\et}} \Omega^{j}_{\cX'/W_n}(\log')_{\et} \, (j=i, i-1)$ induced by the projection $\cX \lra \cX'$, 
we see from the same calculation as that in \eqref{eq:calcdiff} that 
the element $t_1^{-b_1} \omega$ defines a cochain 
in \eqref{eq:decomp-omega0-2} which lifts 
the original cochain $t_1^{-b_1} \omega$ in \eqref{eq:decomp-omega-2}. 
Thus the map \eqref{eq:a-a-surj} is surjective and so we are done. 

(3) \, It is essentially shown in \cite[Lem.~6.1]{BS19} and \cite[Prop.~12]{IY22}, but we repeat the argument. Without loss of generality, we may assume that 
$a=1$. We may assume also that 
there exists a pair $(\cX, \cD)$ and its local coordinate 
$t_1, \dots, t_d$ as in the proof of (1). In this setting, 
it suffices to prove that $\cH^i\left( \frac{\tt^{-\b}\Omega^{\bullet}_{\cX/W_n}(\log)}{ \tt^{-\b+\e_1}\Omega^{\bullet}_{\cX/W_n}(\log)} \right) = 0$ for any $i \in \N$, namely, the complex $\frac{\tt^{-\b}\Omega^{\bullet}_{\cX/W_n}(\log)}{t^{-\b+\e_1}\Omega^{\bullet}_{\cX/W_n}(\log)}$ is acyclic. 
Define $\cX'$ be as in the proof of (1) and 
consider the map 
$$ 
\pi^{j}: \frac{\tt^{-\b}\Omega^{j}_{\cX/W_n}(\log)}{t^{-\b+\e_1}\Omega^{j}_{\cX/W_n}(\log)} \twoheadrightarrow 
t_1^{-b_1}(\tt^{-\b'}\Omega^{j-1}_{\cX'/W_n}(\log')) \wedge \dlog t_1 
\cong t_1^{-b_1}(\tt^{-\b'}\Omega^{j-1}_{\cX'/W_n}(\log')) 
\hookrightarrow 
\frac{\tt^{-\b}\Omega^{j-1}_{\cX/W_n}(\log)}{t^{-\b+\e_1}\Omega^{j}_{\cX/W_n}(\log)},
$$
where the first map is the projection (via the identification \eqref{eq:decomp-omega}), the second isomorphism is the map $\eta \wedge \dlog t_1 \mapsto (-1)^j\eta$ and the third map is the inclusion (via the identification \eqref{eq:decomp-omega}). 
Then, for $t_1^{-b_1}\omega \in 
\frac{\tt^{-\b}\Omega^{j}_{\cX/W_n}(\log)}{t^{-\b+\e_1}\Omega^{j}_{\cX/W_n}(\log)} = t_1^{-b_1} (\tt^{-\b'} \Omega^j_{\cX'/W_n}(\log')) \oplus 
t_1^{-b_1} (\tt^{-\b'} \Omega^{j-1}_{\cX'/W_n}(\log')) \wedge \dlog t_1$ 
with 
$t_1^{-b_1}\omega = t_1^{-b_1}\alpha + t_1^{-b_1}\beta \wedge \dlog t \, 
(\alpha \in \tt^{-\b'} \Omega^j_{\cX'/W_n}(\log'), 
\beta \in \tt^{-\b'} \Omega^{j-1}_{\cX'/W_n}(\log'))$, 
\begin{align*}
d \circ \pi^j(t_1^{-b_1}\omega) & = d((-1)^jt_1^{-b_1}\beta) = (-1)^{j+1}b_1t^{-b_1}\dlog t_1 \wedge \beta + (-1)^jt_1^{-b_1}d\beta \\ & = (-1)^jt_1^{-b_1}d\beta + b_1t^{-b_1}\beta \wedge \dlog t_1, 
\end{align*}
\begin{align*}
\pi^{j+1} \circ d (t_1^{-b_1}\omega) & = 
\pi^{j+1}(
 t_1^{-b_1}d\alpha + t_1^{-b_1}((-1)^{j+1}b_1\alpha + d \beta) 
\wedge \dlog t_1) \\ 
& = b_1 t_1^{-b_1}\alpha + (-1)^{j-1}t_1^{-b_1}d \beta. 
\end{align*}
Thus we see that $d \circ \pi^j + \pi^{j+1} \circ d = b_1 \cdot \id$ and so 
$\{\pi^j\}_j$ gives a homotopy between the zero map and the endomorphism 
$b_1 \cdot \id$ on 
the complex 
$\frac{\tt^{-\b}\Omega^{\bullet}_{\cX/W_n}(\log)}{t^{-\b+\e_1}\Omega^{\bullet}_{\cX/W_n}(\log)}$. 
Since $b_1$ is not divisible by $p$, it implies the 
acyclicity of this complex, as required. 

(4) \, By (1), it suffices to prove that the canonical map 
$\varinjlim_{\b} R^iu_{X.*}\cI_{(X,D)}^{\otimes (-\b)} \lra 
j_*R^iu_{U,*}\cO_{U/W_n}$ is an isomorphism. To check this, 
we may assume that 
there exists a pair $(\cX, \cD)$ and its local coordinate 
$t_1, \dots, t_d$ as in the proof of (1). Then, if we put 
$\cU := \cX \setminus \cD$ and denote the open immersion 
$\cU \hookrightarrow \cX$ also by $j$, the above map is rewritten as 
$\varinjlim_{\b} \cH^i(\tt^{-\b}\Omega^{\bullet}_{\cX/W_n}(\log))
\lra \cH^i(j_*\Omega^{\bullet}_{\cU/W_n})$, which is obviously 
an isomorphism. 
\end{proof}

In the sequel, we denote 
$H^i(((X,D)/W_n)_{\crys}, \cI_{(X,D)}^{\otimes (-\b)})$ simply by 
$H^i_{\crys}(\cI_{(X,D)}^{\otimes (-\b)})$. 

\medskip 


Now let $L \in \Phi$ and let 
$\cP_L$ be the projective system consisting of the triples 
$(X,D,g)$ as in the beginning of this section. 
%
%
If put $U := X \setminus D$ for $(X,D,g) \in \cP_L$, 
we have the canonical isomorphism 
$H^i_{\crys}(U/W_n) \cong 
\Gamma(U, W_n\Omega_U^i)$ of 
Illusie-Raynaud \cite[III (1.5.2)]{IR83}. Hence we have the isomorphism 
$ \varinjlim_{(X,D,g) \in \cP_L} H^i_{\crys}(U/W_n) \cong 
W_n\Omega^i_L$. 

\begin{prop}
For $b \in \N$, 
$\varinjlim_{(X,D,g) \in \cP_L}H^i_{\crys}(\cI_{(X,D)}^{\otimes(-b)})$ is a 
subgroup of $\varinjlim_{(X,D,g) \in \cP_L} H^i_{\crys}(U/W_n) \cong 
W_n\Omega^i_L.$ 
\end{prop}

\begin{proof}
Since we have the equality 
$H^i_{\crys}(\cI_{(X,D)}^{\otimes(-b)}) = 
\Gamma(X, R^iu_{X,*}\cI_{(X,D)}^{\otimes(-b)})$ 
by Proposition \ref{prop:inj}(1) and the equalities 
$H^i_{\crys}(U/W_n) \allowbreak = \Gamma(U, R^iu_{X,*}\cO_{U/W_n}) 
\allowbreak = \Gamma(X, j_*R^iu_{X,*}\cO_{U/W_n})$ (where $j:U \hookrightarrow X$ is the canonical open immersion), the assertion follows from 
Proposition \ref{prop:inj}(4). 
\end{proof}

By this proposition, we can define filtrations on $W_n\Omega^i_L$ in the following way: 

\begin{defn}
$(1)$ \, We define the filtration 
$\{\Fil'_r W_n\Omega^i_L \}_{r \in \Q_{\geq 0}}$ on 
$ W_n\Omega^i_L$ 
as follows: 
\[ 
\Fil'_rW_n\Omega^i_L = 
\left\{ 
\begin{aligned}
& W_n\Omega^i_{O_L}  & (r=0), \\
& 
\varinjlim_{(X,D) \in \cP_L}
H^i_{\crys}(\cI_{(X,D)}^{\otimes(-\lceil r \rceil +1)}) & (r>0).  
\end{aligned}
\right. 
\] 
$(2)$ \, We define the filtration 
$\{\Fil_r W_n\Omega^i_L\}_{r \in \Q_{\geq 0}}$ on 
$ W_n\Omega^i_L$ by 
$$ \Fil_r W_n\Omega^i_L := 
\Fil'_{p\lceil r \rceil}W_n\Omega^i_L. $$ 
\end{defn}

\begin{rem}
The filtration $\Fil$ in (2) is just a renumbering of the filtration $\Fil'$ in (1). 
The filtration $\Fil$ is the one which is compatible with the filtrations of 
Koizumi in \cite{K23}, while the filtration $\Fil'$ is also useful because, if we use it, the proof below looks more parallel to that in Section 4 of \cite{K23}. 
\end{rem}

To prove that the filtrations $\Fil', \Fil$ define ramification filtrations 
on $W_n\Omega^i$, we prepare two lemmas which have essentially appeared in 
\cite{K23}: 



\begin{lem}[{\rm cf. \cite[Lem.~4.2]{K23}}]\label{lem:4.2}
Let $A$ be a smooth algebra over $W_n$ which admits an etale morphism 
$B[t] \lra A$ for some smooth algebra $B$ over $W_n$. Then, for 
$\omega \in \Omega^i_{A/W_n}$, 
it belongs to $t\Omega^i_{A/W_n}(\log t)$ if and only if 
$\omega \wedge \dlog t \in \Omega^{i+1}_{A/W_n}$. 
\end{lem}

\begin{proof}
Note that we have the decompositions 
\begin{align*}
& \Omega^i_{A/W_n} = 
(A \otimes_B \Omega^i_{B/W_n}) 
\oplus (A \otimes_B \Omega^{i-1}_{B/W_n}) \wedge dt, \\
&  \Omega^i_{A/W_n}(\log t) = 
(A \otimes_B \Omega^i_{B/W_n}) 
\oplus (A \otimes_B \Omega^{i-1}_{B/W_n}) \wedge \dlog t. 
\end{align*}
Then, if $\omega \in t\Omega^i_{A/W_n}(\log t)$, it has the form 
$t\omega_1 + \omega_2 \wedge dt$ for some $\omega_1 \in A \otimes_B \Omega^i_{B/W_n}, 
\omega_2 \in A \otimes_B \Omega^{i-1}_{B/W_n}$, and so 
$\omega \wedge \dlog t = \omega_1 \wedge dt \in \Omega^{i+1}_{A/W_n}$. 
Conversely, if $\omega = \omega_1 + \omega_2 \wedge dt$ \, 
$(\omega_1 \in A \otimes_B \Omega^i_{B/W_n}, \omega_2 \in A \otimes_B 
\Omega^{i-1}_{B/W_n})$ 
satisfies the condition $\omega \wedge \dlog t \in \Omega^{i+1}_{A/W_n}$, 
it implies that $t^{-1}\omega_1 \wedge dt \in  \Omega^{i+1}_{A/W_n}$ and so 
$t^{-1}\omega_1 \in A \otimes_B \Omega^i_{B/W_n}$. Hence 
$\omega \in t\Omega^i_{A/W_n}(\log t)$. 
\end{proof}

\begin{lem}[{\rm cf. \cite[Lem.~4.3, 4.4]{K23}}]\label{lem:4.3}

Let $A$ $($resp. $A')$ be a smooth algebra over $W_n$ which admits an etale morphism 
$B[t] \lra A$ $($resp. $B'[t'] \lra A')$ 
for some smooth algebra $B$ $($resp. $B')$ over $W_n$. Also, 
let $f: \Spec A' \lra \Spec A$ be a finite morphism 
such that $f^*t \in A'$ has the form 
$u(t')^e$ for some $u \in (A')^{\times}$ and $e \in \N_{\geq 1}$. 
Also, for $j \in \N$, let 
\begin{align*}
& f_*^j: \Omega^j_{A'[(t')^{-1}]/W_n} \lra \Omega^j_{A[t^{-1}]/W_n} 
 \quad (j \in \N)
\end{align*}
be the pushforward map associated to $f$. 
Then we have the following: \\
$(1)$ \, $f_*^i(t'\Omega^i_{A'/W_n}(\log t')) \subseteq 
t\Omega^i_{A/W_n}(\log t)$. \\ 
$(2)$ \, For $r \in \Q_{>0}$, 
$f_*^i((t')^{-\lceil r \rceil + 1}\Omega^i_{A'/W_n}(\log t')) \subseteq 
t^{-\lceil r/e \rceil + 1}\Omega^i_{A/W_n}(\log t)$. 
\end{lem}

\begin{proof}
(1) \, For $t'\omega \in t'\Omega^i_{A'/W_n}(\log t')$ \, (where $\omega \in 
\Omega^i_{A'/W_n}(\log t')$), 
\begin{align*}
f_*^i(t'\omega) \wedge \dlog t 
& = 
f_*^{i+1}( t'\omega \wedge \dlog (u (t')^e)) \\ 
& = 
f_*^{i+1}(t'\omega \wedge \dlog u) + e 
f_*^{i+1}(t'\omega \wedge \dlog t'). 
\end{align*}
(Here we used the projection formula for the pushforward maps 
$f^j_* \, (j \in \N)$. This follows from Remark \ref{rem:push}.) 
If we write $\omega = \omega_1 + \omega_2 \wedge \dlog t'$ 
with $\omega_1 \in A' \otimes_{B'} \Omega^i_{B'}, \omega_2 \in 
A' \otimes_{B'} \Omega^{i-1}_{B'}$, 
\begin{align*}
& t'\omega \wedge \dlog u = t'\omega_1 \wedge \dlog u 
+ \omega_2 \wedge dt' \wedge \dlog u \in \Omega^{i+1}_{A'}, \\
& t'\omega \wedge \dlog t' = \omega_1 \wedge dt' \in \Omega^{i+1}_{A'}. 
\end{align*}
So $f_*^i(t'\omega) \wedge \dlog t \in f_*^{i+1}(\Omega^{i+1}_{A'/W_n}) 
\subseteq \Omega^{i+1}_{A/W_n}$. Hence, by Lemma \ref{lem:4.2}, 
$f_*^i(t'\omega)$ belongs to $t\Omega^i_{A/W_n}(\log t)$, as required. 

(2) \, For $\omega \in (t')^{-\lceil r \rceil + 1}
\Omega^i_{A'/W_n}(\log t')$, $(t')^{\lceil r \rceil}\omega 
\in t'\Omega^i_{A'/W_n}(\log t')$. Since 
$e\lceil r/e \rceil \geq \lceil r \rceil$, 
$(t')^{e \lceil r/e \rceil}\omega 
\in t'\Omega^i_{A'/W_n}(\log t')$. Then, 
\begin{align*}
t^{\lceil r/e \rceil} f^i_*(\omega) 
& = f^i_*(u^{\lceil r/e \rceil}(t')^{e \lceil r/e \rceil}\omega) \\ 
& \in f^i_*(t'\Omega^i_{A'/W_n}(\log t')) \subseteq 
t\Omega^i_{A/W_n}(\log t)
\end{align*}
by (1). So $f^i_*(\omega) \in 
t^{- \lceil r/e \rceil+1}\Omega^i_{A/W_n}(\log t)$. So we are done. 
\end{proof}

\begin{prop}\label{prop:ramfil}
$(1)$ \, The collection of the filtrations 
$\{\Fil'_r W_n\Omega^i_L \}_{r \in \Q_{\geq 0}}$ on 
$ W_n\Omega^i_L$ for $L \in \Phi$ forms a ramification filtration
on $W_n\Omega^i$. \\
$(2)$ \, The collection of the filtrations 
$\{\Fil_r W_n\Omega^i_L \}_{r \in \Q_{\geq 0}}$ on 
$ W_n\Omega^i_L$ for $L \in \Phi$ forms a ramification filtration
on $W_n\Omega^i$. 
\end{prop}

\begin{proof}
(1) \, By definition, $\Fil'_0 W_n\Omega^i_L = W_n\Omega^i_{O_L}$. 
So it suffices to prove the inclusion 
${\rm Tr}_{L'/L}(\Fil'_r W_n\Omega^i_{L'}) \subseteq 
\Fil'_{r/e} W_n\Omega^i_{L}$ for a finite extension $L \subseteq L'$ 
of geometric henselian valuation fields. 

Let $\cP''_{L.L'}$ be the projective system defined before Lemma 
\ref{lem:cofinal}. Then we have 
$$ W_n\Omega^i_{O_L} = \varinjlim_{\cP''_{L,L'}} 
W_n\Omega^i_{A_0}, \quad 
W_n\Omega^i_{O_{L'}} = \varinjlim_{\cP''_{L,L'}} 
W_n\Omega^i_{A'_0}, $$
where for $(X,D,g,X',D',g',h_0) \in \cP_{L,L'}$, we put 
$X = \Spec A_0$, $X' = \Spec A'_0$. 
Since the pushforward maps 
$h_{0,*}^i: W_n\Omega^i_{A'_0} \lra W_n\Omega^i_{A_0}$ 
of Hodge-Witt sheaves (defined from $h_0: X' \lra X$) 
induce the trace map 
$W_n\Omega^i_{O_{L'}} \lra W_n\Omega^i_{O_{L}}$ which is compatible with 
${\rm Tr}_{L'/L}: W_n\Omega^i_{L'} \lra W_n\Omega^i_{L}$, we see 
the inclusion ${\rm Tr}_{L'/L}(\Fil'_0 W_n\Omega^i_{L'}) \subseteq 
\Fil'_{0} W_n\Omega^i_{L}$. So the claim holds for $r=0$. 

In the following, we assume that $r > 0$. 
Let $\omega$ be an element of $\Fil'_r W_n\Omega^i_{L'}$. 
Then we have 
$$ W_n\Omega^i_{L} = \varinjlim_{\cP''_{L,L'}} 
W_n\Omega^i_{A_0[t_0^{-1}]}, \quad 
\Fil'_rW_n\Omega^i_{L} = 
\varinjlim_{\cP''_{L,L'}}
H^i_{\crys}(\cI_{(X,D)}^{\otimes(-\lceil r \rceil +1)}) 
$$ 
and the similar isomorphisms hold also for $L'$. So 
$\omega$ comes from an element of 
$H^i_{\crys}(\cI_{(X',D')}^{\otimes(-\lceil r \rceil +1)})$ 
for some $(X,D,g,X',D',g',h_0) \in \cP''_{L,L'}$. 
Fix this tuple, and then take a lift 
$h: (\cX',\cD') \lra (\cX,\cD)$ of $h_0$ as in Remark \ref{rem:lift}(5). 
Also, put $U := X \setminus D, U' := X' \setminus D', 
\cU := \cX \setminus \cD, \cU' := \cX' \setminus \cD'$. 
Then 
\begin{align*}
& H^i_{\crys}(\cI_{(X,D)}^{\otimes(-\lceil r/e \rceil +1)}) 
\cong H^i(t^{-\lceil r/e \rceil +1}\Omega^{\bullet}_{A/W_n}(\log t)), \\
& H^i_{\crys}(\cI_{(X',D')}^{\otimes(-\lceil r \rceil +1)}) 
\cong H^i((t')^{-\lceil r \rceil +1}\Omega^{\bullet}_{A'/W_n}(\log t')), \\
& \Gamma(U, W_n\Omega^i_U) \cong 
H^i_{\crys}(U/W_n) \cong H^i(\Omega^{\bullet}_{A[t^{-1}]/W_n}), \\ 
& \Gamma(U', W_n\Omega^i_{U'}) \cong 
H^i_{\crys}(U'/W_n) \cong H^i(\Omega^{\bullet}_{A'[(t')^{-1}]/W_n}) 
\end{align*}
and so we obtain the following commutative diagram: 
\[ 
\xymatrix{
W_n\Omega^i_{L'} \ar[d]^{{\rm Tr}_{L'/L}} & 
\Gamma(U', W_n\Omega^i_{U'}) 
\ar[l] \ar[d]^{h_{0,*}^i} \ar@{-}[r]^{\cong} 
& 
 H^i(\Omega^{\bullet}_{A'[(t')^{-1}]/W_n}) 
\ar[d]^{H^i(h^{\bullet}_*)} & 
H^i((t')^{-\lceil r \rceil +1}\Omega^{\bullet}_{A'/W_n}(\dlog t')) 
\ar[l] \\ 
W_n\Omega^i_{L} & 
\Gamma(U, W_n\Omega^i_{U}) 
\ar[l] \ar@{-}[r]^{\cong} 
& 
 H^i(\Omega^{\bullet}_{A[t^{-1}]/W_n}) & 
H^i(t^{-\lceil r/e \rceil +1}\Omega^{\bullet}_{A/W_n}(\dlog t)). 
\ar[l] 
}
\]
Here $h_{0,*}^i$ is the pushforward map of Hodge-Witt sheaves 
and $H^i(h^{\bullet}_*)$ is the map induced by the pushforward map of 
Hodge sheaves. The commutativity of the left square follows from the 
definition of the map ${\rm Tr}_{L'/L}$. The commutativity of the 
square in the middle is nontrivial, and it follows from Proposition \ref{prop:pushforward}. 

By Lemma \ref{lem:4.3}, there exists a map 
$$ 
H^i((t')^{-\lceil r \rceil +1}\Omega^{\bullet}_{A'/W_n}(\log t'))  
\lra H^i(t^{-\lceil r/e \rceil +1}\Omega^{\bullet}_{A/W_n}(\log t)) 
$$
which fits into the above comutative diagram. Then we see that 
${\rm Tr}_{L'/L}(\omega)$ belongs to the image of 
$H^i(t^{-\lceil r/e \rceil +1}\Omega^{\bullet}_{A/W_n}(\log t))$ and so 
belongs to $\Fil'_{r/e} W_n\Omega^i_{L}$. So we are done. 

(2) \, It suffices to prove the inclusion 
${\rm Tr}_{L'/L}(\Fil_r W_n\Omega^i_{L'}) \subseteq 
\Fil_{r/e} W_n\Omega^i_{L}$ for a finite extension $L \subseteq L'$ 
of geometric henselian valuation fields, when $r>0$. 
Noting that $\Fil_r W_n\Omega^i_{L'} = 
\Fil'_{p \lceil r \rceil} W_n\Omega^i_{L'}$, we see that 
${\rm Tr}_{L'/L}(\Fil_r W_n\Omega^i_{L'}) 
\subseteq \Fil'_{p \lceil r \rceil / e} W_n\Omega^i_{L'}$. 
On the other hand, 
$\Fil_{r/e} W_n\Omega^i_{L} = \Fil'_{p \lceil r/e \rceil} W_n\Omega^i_{L}$. 
So it suffices to prove that 
$\Fil'_{p \lceil r \rceil / e} W_n\Omega^i_{L'} = 
\Fil'_{p \lceil r/e \rceil} W_n\Omega^i_{L}$. 
By Proposition \ref{prop:inj}(3) and the definition of the filtration 
$\Fil'$, $\Fil'_{p \lceil r \rceil / e} W_n\Omega^i_{L'} = 
\Fil'_{p \lceil \lceil r \rceil / e \rceil} W_n\Omega^i_{L'}$. 
So it suffices to prove the equality 
$\left\lceil \dfrac{\lceil r \rceil}{e} \right\rceil = 
\left\lceil \dfrac{r}{e} \right\rceil$. 
The inequality $\left\lceil \dfrac{\lceil r \rceil}{e} \right\rceil \geq
\left\lceil \dfrac{r}{e} \right\rceil$
 is obvious. Conversely, 
\begin{align*}
\left\lceil \dfrac{r}{e} \right\rceil 
\geq \dfrac{r}{e} > \dfrac{\lceil r \rceil - 1}{e} 
= \dfrac{\lceil r \rceil}{e} - \frac{1}{e} 
\geq \left\lceil \dfrac{\lceil r \rceil}{e} \right\rceil 
- \frac{e-1}{e} - \frac{1}{e} = \left\lceil \dfrac{\lceil r \rceil}{e} \right\rceil -1. 
\end{align*}
So we have the equality $\left\lceil \dfrac{\lceil r \rceil}{e} \right\rceil = 
\left\lceil \dfrac{r}{e} \right\rceil$, and the assertion is proved. 
\end{proof}

\begin{prop}\label{prop:ramfilcomp}
$(1)$ \, When $n=1$, our ramfication fitration 
$\{\Fil_r W_n\Omega^i_L \}_{r \in \Q_{\geq 0}}$ coincides with 
the ramification 
filtration $\{\Fil_r \Omega^i(L) \}_{r \in \Q_{\geq 0}}$ of Koizumi {\rm \cite[
Def.~4.1]{K23}}. \\
$(2)$ \, When $i=0$, our ramfication fitration 
$\{\Fil_r W_n\Omega^i_L \}_{r \in \Q_{\geq 0}}$ coincides with 
the ramification 
filtration $\{\Fil_r W_n\cO(L) \}_{r \in \Q_{\geq 0}}$ of Koizumi 
{\rm \cite[Def.~4.8]{K23}}. 
\end{prop}

\begin{proof}
(1) \, 
Let $\cP'_L$ be the projective system defined 
in the beginning of this section and 
take $(X,D,g) \in \cP'_L$ with $A_0, t_0$ as in (A), (B). 
It suffices to see that the Cartier inverse isomorphism 
$$C^{-1}: \Omega^i_{A_0[t_0^{-1}]/k} \os{\cong}{\lra} 
H^i(\Omega^{\bullet}_{A_0[t_0^{-1}]/k})$$ induces the isomorphism 
$t_0^{-\lceil r \rceil + 1}\Omega^i_{A_0/k}(\log t_0) 
\os{\cong}{\lra} 
H^i(t_0^{-p\lceil r \rceil + 1}\Omega^{\bullet}_{A_0/k}(\log t_0))$. 

By the Cartier inverse isomorphism for the log scheme $(X,D)$, 
$C^{-1}$ induces the isomorphism 
$$\Omega^i_{A_0/k}(\log t_0) 
\os{\cong}{\lra} 
H^i(\Omega^{\bullet}_{A_0/k}(\log t_0)).$$ 
Since $C^{-1}$ is Frobenius semilinear, it induces the isomorphism 
$$t_0^{-\lceil r \rceil + 1}\Omega^i_{A_0/k}(\log t_0) 
\os{\cong}{\lra} 
H^i(t_0^{-p(\lceil r \rceil - 1)}\Omega^{\bullet}_{A_0/k}(\log t_0)).$$ 
Since 
$H^i(t_0^{-p\lceil r \rceil + 1}\Omega^{\bullet}_{A_0/k}(\log t_0))$
 is equal to 
$H^i(t_0^{-p(\lceil r \rceil - 1)}\Omega^{\bullet}_{A_0/k}(\log t_0))$
by Proposition \ref{prop:inj}(3), we obtain the required isomorphism. 

(2) \, 
Take the sub projective system $\cQ_L$ of $\cP'_L$ consisting of 
all the triples $(X,D,g)$ as in (A), (B) which satisfies the following 
condition: 

\begin{itemize}
\item 
There exists a lift $(\wt{\cX}, \wt{\cD})$ of $(X,D)$ such that 
$\wt{\cX} = \Spf \wt{A}$ is an affine $p$-adic formal scheme smooth over 
$\Spf W$, $\wt{\cD}$ is a smooth divisor of $\wt{\cX}$ defined by 
a lift $\wt{t} \in \wt{A}$ of 
$t_0 \in A_0$, and that 
there exists a lift $F: \wt{A} \lra \wt{A}$ of Frobenius $F: A_0 \lra A_0$ 
on $A_0$ with $F(\wt{t}) = \wt{t}^p$. 
\end{itemize}

It is easy to see that $\cQ_L$ is cofinal in $\cP'_L$. 
Let $(\cX, \cD), t, F: \cX \lra \cX$ be the mod $p^n$ reduction of 
$(\wt{\cX}, \wt{\cD}), \wt{t}, F: \wt{\cX} \lra \wt{\cX}$ respectively. 
To prove the assertion, it suffices to prove that, for 
$(X,D,g) \in \cQ_L$, the isomorphism 
\begin{equation}\label{eq:ira0}
W_n(A_0[(t_0)^{-1}]) \cong H^0(\Omega^{\bullet}_{A[t^{-1}]/W_n}) 
\end{equation}
of Illusie-Raynaud induces the isomorphism 
\begin{equation}\label{eq:simesui=0}
\{a \in W_n(A_0[(t_0)^{-1}]) \,|\, 
\ul{t_0}^{\lceil r \rceil - 1} F^{n-1}(a) \in W_n(A_0)\} 
\cong 
H^0(t^{-p(\lceil r \rceil - 1)}\Omega^{\bullet}_{A/W_n}(\log t)). 
\end{equation}
Recall that, under the assumption $(X,D,g) \in \cQ_L$, 
the isomorphism is written as 
\begin{equation}\label{eq:ira02}
W_n(A_0[(t_0)^{-1}]) 
\os{\alpha, \cong}{\lra} 
H^0(W_n\Omega^{\bullet}_{A_0[(t_0)^{-1}]}) 
\os{\beta, \cong}{\lla}
H^0(\Omega^{\bullet}_{A[t^{-1}]/W_n}). 
\end{equation}
Here the map $\alpha$ is defined by 
$\alpha(a) := [F^n(a)]$, and 
the map $\beta$ is defined by $\beta([b]) := [(a_0, \dots, a_{n-1})]$, 
where, if $\wt{b}$ is a lift of $b$ to $\widetilde{A}[\wt{t}^{-1}]^{\wedge}$ 
and if we define $\wt{a}_0, \dots, \wt{a}_{n-1} \in \widetilde{A}[\wt{t}^{-1}]^{\wedge}$ by 
\begin{equation}\label{eq:mapbeta}
\wt{a}_0^{p^j} + p \wt{a}_1^{p^{j-1}} + \cdots + p^j \wt{a}_j = F^j \wt{b} 
\quad (0 \leq j \leq n-1), 
\end{equation}
$a_j \, (0 \leq j \leq n-1)$ is $\wt{a}_j$ mod $p$. 

Put $N := \lceil r \rceil - 1$. 
First, let $a \in W_n(A_0[(t_0)^{-1}])$ and consider the condition 
that $\ul{t_0}^NF^{n-1}(a) \in W_n(A_0)$. 
This is equivalent to the condition 
$t_0^{p^jN}a_j^{p^{n-1}} \in A_0 \,(\forall 0 \leq j \leq n-1)$, 
and then is equivalent to the condition 
$t_0^{p^{j+1}N}a_j^{p^n} \in A_0 \,(\forall 0 \leq j \leq n-1)$, 
namely, $\ul{t_0}^{pN}F^n(a) \in W_n(A_0)$, because $A_0$ is normal. 
%
Let $b \in A[t^{-1}]$ be an element with $db=0$, and let 
$\beta([b]) = [a]$ with $a := (a_0,\dots, a_{n-1})$. 
Then, to prove the assertion, it suffices to prove that 
$t^{pN}b \in A$ if and only if 
$a := (a_0,\dots, a_{n-1})$ satisfies the 
condition $\ul{t_0}^{pN} a \in W_n(A_0)$, namely, 
\begin{equation}\label{eq:condition-a} 
\forall 0 \leq j \leq n-1, t_0^{p^{j+1}N}a_j \in A_0. 
\end{equation}

First suppose that $t^{pN}b \in A$. Then there exists a lift 
$\wt{b} \in \wt{A}[\wt{t}^{-1}]^{\wedge}$ with 
$t^{pN}\wt{b} \in \wt{A}$. Then, by \eqref{eq:mapbeta}, 
we see that $p^j \wt{t}^{p^{j+1}N}\wt{a}_j \in \wt{A}$ for $0 \leq j 
\leq n-1$. 
Since $\wt{t}^{p^{j+1}N}\wt{a}_j \in \wt{A}[\wt{t}^{-1}]^{\wedge}$, 
this implies that $\wt{t}^{p^{j+1}N}\wt{a}_j \in \wt{A}$. 
Thus we obtain \eqref{eq:condition-a}. 

Next suppose that $t^{pN}b \not\in A$, and take the minimal integer 
$0 \leq l \leq n-1$ with $t^{pN}b \,{\rm mod}\, p^{l+1} \notin 
A/p^{l+1}A$. Take a lift 
$\wt{b} \in \wt{A}[\wt{t}^{-1}]^{\wedge}$ of $b$. 
Then the equations \eqref{eq:mapbeta} for 
$0 \leq j \leq l-1$ imply that $t^{p^{j+1}N}a_j \in A_0$ for 
$0 \leq j \leq l-1$. Then consider the equation \eqref{eq:mapbeta}
mod $p^{l+1}$ for $j = l$. 
Then we see that, on the left hand side, 
$\wt{a}_0^{p^l} + p \wt{a}_1^{p^{l-1}} + \cdots + p^{l-1} \wt{a}_{l-1}^p$ mod 
$p^{l+1}$ belongs to $A/p^{l+1}A$ if we multiply it by $t^{p^{l+1}N}$. 
On the other hand, $F^l\wt{b}$ mod 
$p^{l+1}$ does not belong to $A/p^{l+1}A$ even 
if we multiply it by $t^{p^{l+1}N}$. Thus we see that 
$p^lt^{p^{l+1}N}\wt{a}_l$ mod $p^{l+1}$ does not belong to 
$A/p^{l+1}A$, namely, $t_0^{p^{l+1}N}a_l \notin A_0$. So we have proved the required assertion. 
\end{proof}

Now we finish the proof of the first part of our main theorem: 

\begin{proof}[Proof of Theorem \ref{thm:main}(1)] 
The collection of the filtrations 
$\{\Fil_r W_n\Omega^i_L \}_{r \in \Q_{\geq 0}}$ on 
$W_n\Omega^i_L$ for $L \in \Phi$ defines a ramification filtration 
on $W_n\Omega^i$ by Proposition \ref{prop:ramfil}(2). 
When $n=1$, it coincides the ramification filtration on 
$\Omega^i$ of Koizumi by Proposition \ref{prop:ramfilcomp}(1), 
and when $i=0$, it coincides the ramification filtration on 
$W_nO$ of Koizumi by Proposition \ref{prop:ramfilcomp}(2). 
\end{proof}

We denote the Nisnevich sheaf on $\MCorQ$ associated to the 
ramification filtration $\{\Fil_r W_n\Omega^i_L \}_{r \in \Q_{\geq 0}}$ 
on $W_n\Omega^i$ by $\ul{\rm M}W_n\Omega^i$. The next proposition gives 
a description of $\ul{\rm M}W_n\Omega^i(\fX)$ when 
$\fX = \in \MCorQls$. 

\begin{prop}[{\rm cf. \cite[Lem.~4.5, 4.13]{K23}}]\label{prop:ls-form}
Let $\fX = (X,D_X) \in \MCorQls$ such that 
$|D_X| = \bigcup_{a=1}^m D_a$ with each $D_a$ smooth and 
$D_X = \sum_{a=1}^m b_a D_a$. 
Put $\b := (b_1, \dots, b_a)$ and 
$\lceil \b \rceil :=  (\lceil b_1 \rceil , \dots, \lceil b_a \rceil)$. 
Then we have 
$$ 
\ul{\rm M}W_n\Omega^i(\fX) = 
\Gamma(X, R^iu_{X,*}\cI_{(X,|D_X|)}^{\otimes (-p\lceil \b \rceil + \1)}), 
$$
where $\1 = (1, \dots, 1)$. 
\end{prop}

\begin{proof}
First we prove that the right hand side is contained in the left hand side. 
Because we may work locally, we may assume that 
$(X,|D_X|)$ is liftable to a pair $(\cX,\cD)$ of an affine 
smooth scheme $\cX = \Spec A$ over 
$W_n$ and a relative simple normal crossing divisor 
$\cD = \bigcup_{a=1}^m \cD_a$ of $\cX$, where $\cD_a$ is a smooth lifting 
of $D_a$ and defined by some $x_a \in A$. 
Put $\cD_{\cX} := \sum_{a=1}^m b_a \cD_a$. 
Then 
\begin{equation}\label{eq:local-ls}
\Gamma(X, R^iu_{X,*}\cI_{(X,|D_X|)}^{\otimes (-p\lceil \b \rceil + \1)})
= \Gamma(\cX, \cH^i(\x^{-p\lceil \b \rceil + \1}
\Omega^{\bullet}_{\cX/W_n}(\log \cD))) = 
\Gamma(\cX, \cH^i(\x^{-p(\lceil \b \rceil - \1)}
\Omega^{\bullet}_{\cX/W_n}(\log \cD))), 
\end{equation}
where we used the multi-index notation $\x^{-p\lceil \b \rceil + \1} := 
\prod_{a=1}^m x_a^{-p\lceil b_a \rceil + 1}$ and 
the last equality follows from Proposition \ref{prop:inj}(3). 
Now suppose given a commutatative diagram \eqref{eq:ramdiag}, and take 
a uniformizer $t_0$ of $O_L$. 
Then there exists some $(Y,E,g) \in \cP'_L$ such that the diagram 
\eqref{eq:ramdiag} is induced from a morphism 
$\varphi_0: (Y,E) \lra (X,|D_X|)$ of log schemes. Then take 
a lift $(\cY,\cE),t$ of $(Y,E),t_0$ as in Remark \ref{rem:lift}(2). 
Then we can take a lift $\varphi: (\cY,\cE) \lra (\cX, \cD)$ of 
$\varphi_0$. In this setting, if we put 
$e_a := v_L(\widetilde{\rho}^*D_a)$, we have 
$\varphi^*x_a = u_a t^{e_a}$ for some $u_a \in 
\Gamma(\cY, \cO_{\cY}^{\times})$. Now assume that 
$[\omega] \in 
\Gamma(\cX, \cH^i(\x^{-p(\lceil \b \rceil - \1)}
\Omega^{\bullet}_{\cX/W_n}(\log \cD)))$. 
Then 
$\x^{p(\lceil \b \rceil -\1)}\omega \in 
\Gamma(\cX, \Omega^{i}_{\cX/W_n}(\log \cD))$ and so 
$\varphi^*(\x^{p(\lceil \b \rceil -1)}\omega) 
= u t^{\sum_a pe_a(\lceil b_a \rceil -\1)}\varphi^*\omega 
\in \Gamma(\cY, \Omega^{i}_{\cY/W_n}(\log \cE))$, 
where $u$ is some unit. Since 
$\sum_a e_a(\lceil b_a \rceil - 1) < \sum_a e_ab_a \leq 
\lceil \sum_a e_ab_a \rceil$, we have 
$\sum_a e_a(\lceil b_a \rceil - 1) \leq \lceil \sum_a e_ab_a \rceil - 1$ 
and so $t^{-p(\lceil \sum_a e_a b_a \rceil -1)}\varphi^*\omega 
\in \Gamma(\cY, \Omega^{i}_{\cY/W_n}(\log \cE))$. Thus 
$\varphi^*[\omega] \in \Gamma(\cY, \cH^i(
t^{-p(\lceil \sum_a e_a b_a \rceil -1)}\Omega^{\bullet}_{\cY/W_n}(\log \cE)))
= 
\Gamma(\cY, \cH^i(
t^{-p\lceil \sum_a e_a b_a \rceil +1}\Omega^{\bullet}_{\cY/W_n}(\log \cE))))
$ 
and so $\rho^*[\omega] \in \Fil_{\lceil \sum_a e_a b_a \rceil}W_n\Omega^i_L
= \Fil_{v_L(\widetilde{\rho}^*D_X)}W_n\Omega^i_L$. 
Hence $[\omega]$ belongs to $\ul{\rm M}W_n\Omega^i(\fX)$ by definition of 
$\ul{\rm M}W_n\Omega^i$. 

Next we prove that the left hand side is contained in the right hand side. 
Because we may work locally, we can keep the notation in the previous 
paragraph. Let $\omega$ be an element in $\ul{\rm M}W_n\Omega^i(\fX)$ which 
does not belong to the group \eqref{eq:local-ls}. Take 
$\c = (c_1, \dots, c_m) \in \N^m$ with 
$c_j \geq \lceil b_j \rceil \, (\forall j)$ such that 
$$ \omega \in 
\Gamma(X, R^iu_{X,*}\cI_{(X,|D_X|)}^{\otimes (-p\c + \1)})
= 
\Gamma(X, R^iu_{X,*}\cI_{(X,|D_X|)}^{\otimes (-p(\c - \1))}) 
= 
\Gamma(\cX, \cH^i(\x^{-p(\c - \1)}
\Omega^{\bullet}_{\cX/W_n}(\log \cD))) 
$$ 
and that $|\c|$ is minimal among the indices satisfying this condition. 
Also, take an index $j_0$ with $c_{j_0} > \lceil b_{j_0} \rceil$. Then 
$\omega$ does not belong to 
$$\Gamma(\cX, \cH^i(\x^{-p(\c - \e_{j_0} - \1)}
\Omega^{\bullet}_{\cX/W_n}(\log \cD))) 
= 
\Gamma(\cX, \cH^i(\x^{-p(\c - \1) + \e_{j_0}}
\Omega^{\bullet}_{\cX/W_n}(\log \cD))).$$ 
Since we have the exact sequence 
\begin{align*}
\Gamma(\cX, \cH^i(\x^{-p(\c - \1) + \e_{j_0}} \Omega^{\bullet}_{\cX/W_n}(\log \cD)))
& \lra  
\Gamma(\cX, \cH^i(\x^{-p(\c - \1)} \Omega^{\bullet}_{\cX/W_n}(\log \cD))) 
\\ & \lra  
\Gamma\left(\cX, \cH^i\left(
\dfrac{\x^{-p(\c - \1)}\Omega^{\bullet}_{\cX/W_n}(\log \cD)}{\x^{-p(\c - \1) + \e_{j_0}} \Omega^{\bullet}_{\cX/W_n}(\log \cD)}  \right) \right), 
\end{align*}
the image $\omega_1$ of $\omega$ in 
$\Gamma\left(\cX, \cH^i\left(
\frac{\x^{-p(\c - \1} \Omega^{\bullet}_{\cX/W_n}(\log \cD)}{\x^{-p(\c - \1) + \e_{j_0}} \Omega^{\bullet}_{\cX/W_n}(\log \cD)}  \right) \right)$ is nonzero. 
To lighten the notation, put $\cF_n^{\bullet} := \frac{\x^{-p(\c - \1} \Omega^{\bullet}_{\cX/W_n}(\log \cD)}{\x^{-p(\c - \1) + \e_{j_0}} \Omega^{\bullet}_{\cX/W_n}(\log \cD)}$ and for $0 \leq n' \leq n$, let 
$\cF_{n'} := \cF_n \otimes_{W_n} W_{n'}$. Let $1 \leq n_0 \leq n$ be the minimal integer such that 
the image $\omega_2$ of $\omega_1$ in 
$\Gamma(\cX, \cH^i(\cF^{\bullet}_{n_0}))$ is nonzero. Since we have the exact sequence 
$$ 
\Gamma(X, \cH^i(\cF^{\bullet}_{1})) \os{\alpha}{\lra}
\Gamma(\cX, \cH^i(\cF^{\bullet}_{n_0})) \os{\beta}{\lra}
\Gamma(\cX, \cH^i(\cF^{\bullet}_{n_0-1})), 
$$ 
we see that there exists a nonzero element $\omega_3 \in 
\Gamma(X, \cH^i(\cF^{\bullet}_{1}))$ with $\alpha(\omega_3) = \omega_2$. 
Then note that we have the isomorphism of complexes 
\begin{align*}
\cF^{\bullet}_1 = 
\dfrac{\x^{-p(\c - \1)}\Omega^{\bullet}_{X/k}(\log D)}{\x^{-p(\c - \1) + \e_{j_0}} \Omega^{\bullet}_{X/k}(\log D)}  
& = 
x_{j_0}^{-c_{j_0}} (\x^{-p(\c' - \1')} \Omega^{\bullet}_{D_{j_0}/k}(\log D')) \oplus 
x_{j_0}^{-c_{j_0}} (\x^{-p(\c' - \1')} \Omega^{\bullet-1}_{D_{j_0}/k}(\log D')) \wedge \dlog x_{j_0} \\ 
& \cong \x^{-p(\c' - \1')} \Omega^{\bullet}_{D_{j_0}/k}(\log D') \oplus 
\x^{-p(\c' - \1')} \Omega^{\bullet-1}_{D_{j_0}/k}(\log D') \\
& \cong \Omega^{\bullet}_{D_{j_0}/k}(\log D') \oplus \Omega^{\bullet-1}_{D_{j_0}/k}(\log D')
\end{align*}
(where $\x^{-p(\c' - \1')} := \prod_{\substack{1 \leq a \leq m \\ a \not= j_0}}
x_a^{-p(c_a - 1)}$ and 
$D' = \bigcup_{a \not= j_0}(D_a \cap X')$): The isomorphism on each degree follows from the Zariski version of \eqref{eq:decomp-omega-2}, and the isomorphism as complexes follows from the computation as in \eqref{eq:calcdiff} and the fact that the complexes above are of characteristic $p$. Then, by Cartier inverse isomorphism for log schemes, we have an isomorphism 
$$\cH^i(\cF^{\bullet}_{1}) \cong \Omega^{i}_{D_{j_0}/k}(\log D') \oplus \Omega^{i-1}_{D_{j_0}/k}(\log D').$$
%
%
Let $U \lra X$ be any affine etale morphism whose image contains 
the generic point $\xi$ of $D_{j_0}$ and which does not meet 
any $D_a$ for $a \not= j_0$, and put $U' := U \times_X D_{j_0}$, 
$D_U := U \times_X |D_X| = U \times_X D_{j_0}$. 
Then 
$$\cH^i(\cF^{\bullet}_{1}|_U) \cong \Omega^{i}_{U'/k} \oplus \Omega^{i-1}_{U'/k} $$
and $\omega_3$ is nonzero in it. 
Then, by construction, we see that the image 
of $\omega$ in 
$$ 
\Gamma(U, R^iu_{U,*}\cI_{(U,D_U)}^{\otimes (-pc_{j_0} + 1)})
= 
\Gamma(U, R^iu_{U,*}\cI_{(U,D_U)}^{\otimes (-p(c_{j_0} - 1))}) 
$$ 
does not belong to 
$
\Gamma(U, R^iu_{U,*}\cI_{(U,D_U)}^{\otimes (-p(c_{j_0} - 1)+1)}) 
\supseteq 
\Gamma(U, R^iu_{U,*}\cI_{(U,D_U)}^{\otimes (-p\lceil b_{j_0} \rceil +1)})$. 
This implies that, if we define $L$ by $L = \Frac \cO_{X,\xi}^{\rm h}$, 
the image of $\omega$ in $W_n\Omega^i_L$ does not belong to 
$\Fil_{b_{j_0}}W_n\Omega^i_L$ and so $\omega$ does not belong to 
$\ul{\rm M}W_n\Omega^i(\fX)$. This is a contradiction. So 
$\omega$ belongs to the group \eqref{eq:local-ls} and so we are done. 
\end{proof}

Now we prove the second part of our main theorem. 

\begin{proof}[Proof of Theorem \ref{thm:main}(2)]
The property (3) in Theorem \ref{thm:koizumi1} follows immediately from 
Proposition \ref{prop:ls-form}. The property (2) in Theorem \ref{thm:koizumi1} 
follows from Proposition \ref{prop:ls-form} and the quasi-coherence of 
the sheaves of the form 
$R^iu_{X,*}\cI_{(X,|D_X|)}^{\otimes (-p\lceil \b \rceil + \1)}$, which is proved in Proposition \ref{prop:inj}(1). 

We prove that the sheaf $\ul{\rm M}W_n\Omega^i$ satisfies the property (1) 
in Theorem \ref{thm:koizumi1}(1). Since we may work locally, we may assume 
that 
$(X,|D_X|)$ is liftable to a pair $(\cX,\cD)$ of an affine 
smooth scheme $\cX = \Spec A$ over 
$W_n$ and a relative simple normal crossing divisor 
$\cD = \bigcup_{a=1}^m \cD_a$ of $\cX$, where $\cD_a$ is a smooth lifting 
of $D_a$ and defined by some $x_a \in A$. 
In this setting, using the isomorphism 
$$ 
R^iu_{X,*}\cI_{(X,|D_X|)}^{\otimes (-p\c + \1)}
= 
R^iu_{X,*}\cI_{(X,|D_X|)}^{\otimes (-p(\c - \1))} 
= 
\cH^i(\x^{-p(\c - \1)}\Omega^{\bullet}_{\cX/W_n}(\log \cD)) 
$$ 
and long exact sequences associated to the exact sequences 
\begin{align*}
0 \lra \x^{-p(\c - \1) + \e_{j_0}} \Omega^{\bullet}_{\cX/W_n}(\log \cD)
& \lra  
\x^{-p(\c - \1)} \Omega^{\bullet}_{\cX/W_n}(\log \cD)
\lra 
\dfrac{\x^{-p(\c - \1)}\Omega^{\bullet}_{\cX/W_n}(\log \cD)}{\x^{-p(\c - \1) + \e_{j_0}} \Omega^{\bullet}_{\cX/W_n}(\log \cD)} \lra 0, 
\end{align*}
\begin{align*}
0 \lra \cF^{\bullet}_1 \lra \cF^{\bullet}_n \lra \cF^{\bullet}_{n-1} \lra 0
\end{align*}
(where $\cF_n^{\bullet} := \frac{\x^{-p(\c - \1} \Omega^{\bullet}_{\cX/W_n}(\log \cD)}{\x^{-p(\c - \1) + \e_{j_0}} \Omega^{\bullet}_{\cX/W_n}(\log \cD)}$ and 
$\cF_{n'} := \cF_n \otimes_{W_n} W_{n'}$) and the isomporphism of complexes 
\begin{align*}
\cF_1 & = 
x_{j_0}^{-c_{j_0}} (\x^{-p(\c' - \1')} \Omega^{\bullet}_{D_{j_0}/k}(\log D')) \oplus 
x_{j_0}^{-c_{j_0}} (\x^{-p(\c' - \1')} \Omega^{\bullet-1}_{D_{j_0}/k}(\log D')) \wedge \dlog x_{j_0} \\ 
& \cong \x^{-p(\c' - \1')} \Omega^{\bullet}_{D_{j_0}/k}(\log D') \oplus 
\x^{-p(\c' - \1')} \Omega^{\bullet-1}_{D_{j_0}/k}(\log D') \\
& \cong \Omega^{\bullet}_{D_{j_0}/k}(\log D') \oplus \Omega^{\bullet-1}_{D_{j_0}/k}(\log D')
\end{align*}
(where $\x^{-p(\c' - \1')} := \prod_{\substack{1 \leq a \leq m \\ a \not= j_0}}
x_a^{-p(c_a - 1)}$ and 
$D' = \bigcup_{a \not= j_0}(D_a \cap X')$) appeared in the proof of Propsition \ref{prop:ls-form} and usual weight filtration on log de Rham complexes, we can reduce the claim to the case $n=1$ and $D_X = \emptyset$, 
namely, it suffices to prove the quasi-isomorphism 
$$ \cH^i(\Omega^{\bullet}_{X/k}) \lra 
R\pi_*\cH^i(\Omega^{\bullet}_{X \times \ol{\square}/k}), $$
where $\pi$ is the projection $X \times \ol{\square} \lra X$. 
By Cartier isomorphism, the above map is rewritten as 
$$ \Omega^i_{X/k} \lra 
R\pi_* \Omega^i_{X \times \ol{\square}/k}, $$
and we have 
\begin{align*}
R\pi_* \Omega^i_{X \times \ol{\square}/k}
& = R\pi_*\pi^*\Omega^i_{X/k} \oplus 
R\pi_*(\pi^* \Omega^{i-1}_{X/k} \otimes \Omega^1_{X \times \ol{\square}/X}) \\ 
& = (\Omega^i_{X/k} \otimes R \pi_* \cO_{X \times \Pr^1_k})
\oplus (\Omega^{i-1}_{X/k} \otimes R \pi_* \Omega^1_{X \times \ol{\square}/X})
= \Omega^i_{X/k}. 
\end{align*}
So the property (1) is proved and thus the proof is finished. 
\end{proof}

\appendix

\section{Comatibility of pushforward maps}

In this appendix, we prove a compatibility result between 
the pushforward maps of Hodge-Witt sheaves and those of 
de Rham complexes. 

Let $f: \cX' \lra \cX$ be a finite surjective 
morphism between connected smooth affine schemes 
over $W_n$ and let $f_0: X' \lra X $ be the mod $p$ reduction of $f$. 
Also, let $i$ be a nonnegative integer. 
Then, as we reviewed in Section 1, we have the pushforward map 
\begin{equation}\label{eq:push-hw}
f_{0,*}^i: W_n\Omega^i_{X'} \lra W_n\Omega_X^i 
\end{equation}
of Hodge-Witt sheaves. 
On the other hand, we have also the pushforward map 
$$f_*^i: \Omega^i_{\cX'/W_n} \lra \Omega^i_{\cX/W_n}$$ of 
Hodge sheaves over $W_n$: It is defined as the sheafification of 
the pushforward map in \cite[Def.~2.3.2]{CR11}. 
(In \cite{CR11}, the pushforward maps are defined for varieties over a field using the duality theory of Hartshorne \cite{H66} and 
Conrad \cite{C00}, \cite{C11}. But the arguments in Section 2 of 
\cite{CR11} works also for schemes of finite type over $W_n$ 
verbatim.) 
By \cite[Prop.~2.3.3 (3)]{CR11}, $f_*^i$ is equal to the map 
$\tau_f$ in \cite[Prop.~2.2.23]{CR11}, and then the explicit 
description of the map $\tau_f$ in  \cite[(2.2.24)]{CR11} shows that 
the compatibility with the differentials 
$d \circ f_*^i = f_*^{i+1} \circ d$ holds. Hence we have the pushforward map 
$$ f_*^{\bullet}: \Omega^{\bullet}_{\cX'/W_n} \lra \Omega^{\bullet}_{\cX/W_n}$$ of de Rham complexes and thus the pushforward map 
\begin{equation}\label{eq:push-drc}
\cH^i(f_*^{\bullet}): \cH^i(\Omega^{\bullet}_{\cX'/W_n}) \lra 
\cH^i(\Omega^{\bullet}_{\cX/W_n}) 
\end{equation}
of de Rham cohomology sheaves. 

\begin{rem}\label{rem:push}
Because $f_*^i$ is equal to the map 
$\tau_f$ in \cite[Prop.~2.2.23]{CR11}, it satisfies 
the projection formula \cite[Prop.~2.2.23]{CR11}. 
We use this fact in the main body of the article.
\end{rem}

Now recall that, in the situation above, there exist isomorphisms 
\begin{equation}\label{eq:ir}
\cH^i(\Omega^{\bullet}_{\cX/W_n}) \cong 
R^iu_{X,*}\cO_{X/W_n} \cong 
W\Omega^i_{X}, 
\end{equation}
\begin{equation}\label{eq:ir'}
\cH^i(\Omega^{\bullet}_{\cX'/W_n}) \cong
R^iu_{X',*}\cO_{X'/W_n} \cong 
W\Omega^i_{X'}
\end{equation}
due to Illusie-Raunaud \cite[III (1.5.2), (1.5.3)]{IR83}, where $u_X: (X/W_n)_{\crys}^{\sim} \lra X_{\Zar}^{\sim}, u_{X'}: (X'/W_n)_{\crys}^{\sim} \lra {X'}_{\Zar}^{\sim}$ 
are projections from crystalline topoi to Zariski topoi. 
We prove the following compatibility result of 
the pushforward maps \eqref{eq:push-hw}, \eqref{eq:push-drc} and 
isomorphisms \eqref{eq:ir}, \eqref{eq:ir'}: 

\begin{prop}\label{prop:pushforward}
Let the notations be as above. Then we have the commutative diagram 
\[ 
\xymatrix{
\cH^i(\Omega^{\bullet}_{\cX'/W_n}) 
\ar[d]^{\cong} \ar[r]^{\cH^i(f_*^{\bullet})} & 
\cH^i(\Omega^{\bullet}_{\cX/W_n}) \ar[d]^{\cong} \\
W_n\Omega^i_{X'} \ar[r]^{f_{0,*}^i} & 
W_n\Omega_X^i, 
}
\]
where the vertical arrows are the isomorphisms \eqref{eq:ir}, \eqref{eq:ir'}. 
\end{prop}

\begin{proof}
Take a factorization 
$\cX' \overset{\iota}{\hookrightarrow} \cP = {\mathbb{P}}^d_{\cX} 
\overset{\pi}{\lra} \cX$ of $f$ where $\iota$ is a closed immersion 
and $\pi$ is the natural projection, and let 
$X' \overset{\iota_0}{\hookrightarrow} P = {\mathbb{P}}^d_{X} 
\overset{\pi_0}{\lra} X$ be the mod $p$ reduction of it. 
By the definition and the functoriality of pushforward maps 
in \cite[Def.~2.3.2, Prop.~2.3.3]{CR11}, 
the map $f_*^i$ is written as the composite 
\begin{equation*}
f_*\Omega^{i}_{\cX'/W_n} = 
\pi_*\iota_*\Omega^{i}_{\cX'/W_n} 
\os{\pi_*\psi_1^i}{\lra}
\pi_*(\cH^d_{\cX'}(\Omega^{i+d}_{\cP/W_n}))
\os{\pi_*\psi_2^i}{\lra} \pi_*(\Omega^{i+d}_{\cP/W_n}[d]) 
\os{\psi_3^i}{\lra} \Omega^{i}_{\cX/W_n}, 
\end{equation*}
where the maps $\psi_j^i \, (j=1,2,3)$ are described as follows: 

\smallskip 

\noindent
(1) \, 
%
If we work locally and assume that $\iota$ is defined by 
a regular sequence $t_1, \dots, t_d$, 
the map $\psi_1^i$ is described as 
$$ 
\omega \mapsto (-1)^d ({\rm dlog}\,t_1 \wedge \cdots \wedge {\rm dlog}\,t_d 
\wedge \omega) $$
(\cite[Prop.~2.2.19]{CR11}). Here the right hand side is 
regarded as an element in 
$\cH^d_{\cX'}(\Omega^{i+d}_{\cP/W_n})$ via the isomorphism 
$$ \cH^d_{\cX'}(\Omega^{i+d}_{\cP/
W_n}) \cong 
\Omega^{i+d}_{\cU_{[1,d]}/W_n}/
\sum_{\substack{J \subseteq [1,d] \\ \# J = d-1}} 
\Omega^{i+d}_{\cU_{J}/W_n}, $$
where, for $J \subseteq [1,d]$, $\cU_J$ is the open subscheme 
$\{ \forall j \in J, t_j \not= 0\}$ of $\cP$. 
\\
(2) \, The map $\psi_2^i$ is induced by the canonical map 
$$ \cH^d_{\cX'}(\Omega^{i+d}_{\cP/W_n}) = 
R\ul{\Gamma}_{\cX'}(\Omega^{i+d}_{\cP/W_n})[d] \lra 
\Omega^{i+d}_{\cP/W_n}[d].$$
(3) \, The map $\psi_3^i$ 
is an isomorphism. (This is proved in the same way as 
\cite[Lem.~2.4.3]{CR11}.) By the functoriality of pushforward maps,  
if we take a section $s: \cX \lra \cP$ of $\pi$, it is the inverse of the map 
$\pi_*\phi_3^i: \Omega^i_{\cX/W_n} = \pi_*s_*\Omega^i_{\cX/W_n} \lra 
\pi_*(\Omega^{i+d}_{\cP/W_n}[d])$ 
which is defined as the map $\pi_*\psi_2^i \circ \pi_*\psi_1^i$ for the closed immersion $s$. 

\smallskip 

We see by the explicit description of the map $\psi_1^i$ given above 
the map $\psi_1^i$ is compatible with the differential $d$ of de Rham complex. The map $\psi_3^i$ is also compatible with the differential $d$ 
because so is $(\psi_3^i)^{-1} = \pi_*\phi_3^i$, 
and the map $\psi_2^i$ is 
obviously compatible with the differential $d$. Thus 
the map $\cH^i(f_*^{\bullet})$ is written as the composite 
\begin{align*}
f_*\cH^i(\Omega^{\bullet}_{\cX'/W_n}) = 
\pi_*\iota_*\cH^i(\Omega^{\bullet}_{\cX'/W_n}) 
& \os{\pi_*\cH^i(\psi_1^{\bullet})}{\lra}
\pi_*(\cH^d_{\cX'}(\cH^{i+d}(\Omega^{\bullet}_{\cP/W_n}))) \\ 
& \os{\pi_*\cH^i(\psi_2^{\bullet})}{\lra} 
\pi_*(\cH^{i+d}(\Omega^{\bullet}_{\cP/W_n})[d]) 
\os{\cH^i(\psi_3^{\bullet})}{\lra} \cH^i(\Omega^{\bullet}_{\cX/W_n}), 
\end{align*}
where $\cH^i(\psi_1^{\bullet})$ has the description 
$$ 
[\omega] \mapsto [(-1)^d ({\rm dlog}\,t_1 \wedge \cdots \wedge {\rm dlog}\,t_d 
\wedge \omega)] $$
as in (1) above, where $[-]$ means the cohomology class, and 
$\cH^i(\psi_3^{\bullet})$ is the inverse of the map 
$$ \pi_*\cH^i(\phi_3^{\bullet}): 
\cH^i(\Omega^{\bullet}_{\cX/W_n}) = 
\pi_*s_*\cH^i(\Omega^{\bullet}_{\cX/W_n}) \lra 
\pi_*\cH^{i+d}(\Omega^{\bullet}_{\cP/W_n}). 
$$ 

On the other hand, the pushforward map $f_{0,*}^i$ has a similar description: 
By the functoriality of pushforward map (\cite[Prop.~2.3.3]{CR12}), 
the maps $f_{0,*}^i$ is written as the composite 
\begin{align*}
f_{0,*}W_n\Omega^{i}_{X'} = 
\pi_{0,*}\iota_{0,*}W_n\Omega^{i}_{X'} 
& \os{\pi_{0,*}\psi_{0,1}^i}{\lra}
\pi_{0,*}(\cH^d_{X'}(W_n\Omega^{i+d}_{P})) \\ & 
\os{\pi_{0,*}\psi_{0,2}^i}{\lra} \pi_{0,*}(W_n\Omega^{i+d}_{P}[d]) 
\os{\psi_{0,3}^i}{\lra} W_n\Omega^{i}_{X}, 
\end{align*}
where the maps $\psi_{0,j}^i \, (j=1,2,3)$ are described as follows: 

\smallskip 

\noindent
(1)' \, 
If we work locally and assume that $\iota_0$ is defined by 
regular sequence $\ol{t}_1, \dots, \ol{t}_d$, $\psi_{0,1}^i$ is described as 
$$ 
\omega \mapsto (-1)^d ({\rm dlog}\,\ul{\ol{t}_1} \wedge \cdots \wedge 
{\rm dlog}\,\ul{\ol{t}_d}
\wedge \omega), $$
where $\ul{\ol{t}_j}$ is the Teichm\"uller lift of $\ol{t}_j$ 
(\cite[Prop.~2.4.1]{CR12}). Here the right hand side is 
regarded as an element in 
$\cH^d_{X'}(W_n\Omega^{i+d}_{P})$ via the isomorphism 
$$ \cH^d_{X'}(W_n\Omega^{i+d}_{P}) \cong 
W_n\Omega^{i+d}_{U_{[1,d]}}/
\sum_{\substack{J \subseteq [1,d] \\ \# J = d-1}} 
W_n\Omega^{i+d}_{U_{J}}, $$
where, for $J \subseteq [1,d]$, $U_J$ is the open subscheme 
$\{ \forall j \in J, t_j \not= 0\}$ of $P$. \\
(2)' \, The map $\psi_{0,2}^i$ is induced by the canonical map 
$$ \cH^d_{X'}(W_n\Omega^{i+d}_{P}) = 
R\ul{\Gamma}_{X'}(W_n\Omega^{i+d}_{P})[d] \lra 
W_n\Omega^{i+d}_{P}[d].$$
(See \cite[Prop.~2.4.1]{CR12} for the first equality.) \\
(3)' \, The map $\psi_{0,3}^i$ 
is an isomorphism (\cite[Lem.~2.4.3]{CR12}).  
By the functoriality of pushforward maps,  
if we take a section $s_0: X \lra P$ of $\pi_0$, it is the inverse of the map 
$\pi_{0,*}\phi_{0,3}^i: W_n\Omega^i_{X} = \pi_{0,*}s_{0,*}W_n\Omega^i_{X} \lra 
\pi_{0,*}(W_n\Omega^{i+d}_{P}[d])$ 
which is defined as the map $\pi_{0,*}\psi_{0,2}^i \circ 
\pi_{0,*}\psi_{0,1}^i$ 
for the closed immersion $s_0$. 

\smallskip 

By these descriptions, to prove the proposition, it suffices to prove the 
compatibility of the map $\cH^i(\psi_2^{\bullet} \circ \psi_1^{\bullet})$ 
and the map $\psi_{0,2} \circ \psi_{0,1}$, and that of 
the map $\cH^i(\phi_3^{\bullet})$ and the map $\phi_{0,3}$. 
Hence it suffices to prove the compatibility of the pushforward maps 
for regular closed immersions via the isomphisms of Illusie-Raynaud. 
Moreover, since this claim is Zariski local 
and since the pushforward maps have functoriality, 
we may work with affine schemes and we may assume that 
the regular closed immersions in consideration are 
of codimension $1$ defined by a single section. 
So, in the following, 
assume we are given a regular closed immersion 
$\iota: \cX \hookrightarrow \cP$ of affine smooth schemes over $W_n$ 
defined by a section $t$ and let $\iota_0:X \hookrightarrow P, \ol{t}$ 
be the mod $p$ reduction of $\iota, t$ respectively and we prove the 
compatibility of the pushforward map 
$$ \cH^i(\iota_*^{\bullet}): 
\iota_*\cH^i(\Omega^{\bullet}_{\cX}) 
\os{\cH^i(\psi_1^{\bullet})}{\lra}
\cH^1_{\cX}(\cH^{i+1}(\Omega^{\bullet}_{\cP})) 
\os{\cH^i(\psi_2^{\bullet})}{\lra} 
\cH^{i+1}(\Omega^{\bullet}_{\cP})[1] $$
and the pushforward map 
$$ 
\iota_{0,*}: 
\iota_{0,*}W_n\Omega_{X}^i 
\os{\psi_{0,1}^i}{\lra} 
\cH_X^1(W_n\Omega^{i+1}_P) 
\os{\psi_{0,2}^i}{\lra} W_n\Omega^{i+1}_P[1] $$
via the isomorphism \eqref{eq:ir} and the corresponding isomorphism 
\begin{equation}\label{eq:irp}
\cH^i(\Omega^{\bullet}_{\cP/W_n}) \cong 
R^iu_{P,*} \cO_{P/W_n} \cong 
W\Omega^i_{P}
\end{equation}
for $P$. (Here, as before, $u_P:(P/W_n)_{\crys}^{\sim} \lra P_{\Zar}^{\sim}$ 
denotes the 
projection from the crystalline topos to the Zariski topos.) 
Moreover, since the compatibility of the maps $\cH^i(\psi_2^{\bullet})$, 
$\psi_{0,2}^i$ follows from the functoriality, we are reduced to proving 
the compatibility of the maps $\cH^i(\psi_1^{\bullet})$, 
$\psi_{0,1}^i$. 
To prove it, we first prove the following claim: 

\smallskip

\noindent
{\bf claim.} \, 
Let $\iota': \cX' \lra \cP'$ be another closed immersion of 
affine schemes over $W_n$ defined by a section $t'$ such that 
the mod $p$ reduction of $\iota', t'$ is equal to $\iota_0, \ol{t}$ 
respectively.  Also, let $\cD$ be the PD-envelope of 
$X$ in $\cX \times_{W_n} \cX'$ and let $\cD_P$ be the 
PD-envelope of $P$ in $\cP \times_{W_n} \cP'$. Then we have the 
following commutative diagram:
\begin{equation}\label{eq:comprci}
\xymatrix{
\iota_*\cH^i(\Omega^{\bullet}_{\cX/W_n}) 
\ar[d]^{\cong} \ar[r]^-{\cH^i(\psi_1^{\bullet})} 
& 
\cH^1_{\cX}(\cH^{i+1}(\Omega^{\bullet}_{\cP/W_n}))
\ar[d]^{\cong} 
\\ 
\iota_*\cH^i(\Omega^{\bullet}_{\cD/W_n}) 
& 
\cH^1_{\cD}(\cH^{i+1}(\Omega^{\bullet}_{\cD_P/W_n})) 
\\
\iota_*\cH^i(\Omega^{\bullet}_{\cX'/W_n}) 
\ar[u]^{\cong} \ar[r]^-{\cH^i({\psi_1^{\bullet}}')} 
& 
\cH^1_{\cX'}(\cH^{i+1}(\Omega^{\bullet}_{\cP'/W_n}))
\ar[u]^{\cong} 
}
\end{equation}
%
Here, $\Omega^{\bullet}_{\cD/W_n}$, 
$\Omega^{\bullet}_{\cD_P/W_n}$ denote the PD-de Rham complex 
for $\cD, \cD_P$ respectively, 
${\psi_1^{\bullet}}'$ is the map 
$\psi_j^{\bullet}$ for $\cX' \hookrightarrow \cP'$ and 
the vertical isomorphisms are 
induced by the projections $\cD \lra \cX, \cD \lra \cX', 
\cD_P \lra \cP, \cD_P \lra \cP'$. 

\begin{proof}[Proof of claim]
An element 
$[\alpha] \in \cH^i(\Omega^{\bullet}_{\cX/W_n})$ is sent by 
$\cH^i(\psi_1^{\bullet})$ to 
$[\alpha \wedge \dlog t] \in \cH^1_{\cX}(\cH^{i+1}(\Omega^{\bullet}_{\cP/W_n}))  = \cH^{i+1}(\Omega^{\bullet}_{\cU/W_n}/\Omega^{\bullet}_{\cP/W_n})$ 
(where $\cU = \cP \setminus \cX$), and 
$[\alpha'] \in \cH^i(\Omega^{\bullet}_{\cX'/W_n})$ is sent by 
$\cH^i({\psi_1^{\bullet}}')$ to 
$[\alpha' \wedge \dlog t'] \in \cH^1_{\cX}(\cH^{i+1}(\Omega^{\bullet}_{\cP'/W_n}))  = \cH^{i+1}(\Omega^{\bullet}_{\cU'/W_n}/\Omega^{\bullet}_{\cP'/W_n})$ 
(where $\cU' = \cP' \setminus \cX'$). 
So we need to prove that, if $[\alpha]$ and $[\alpha']$ agree in 
$\cH^i(\Omega^{\bullet}_{\cD/W_n})$, 
$[\alpha \wedge \dlog t]$ and $[\alpha' \wedge \dlog t']$ agree 
in 
$\cH^1_{\cX}(\cH^{i+1}(\Omega^{\bullet}_{\cD_P/W_n}))  = \cH^{i+1}(\Omega^{\bullet}_{\cU''/W_n}/\Omega^{\bullet}_{\cD_P/W_n})$, where 
$\cU'' = \cD_P \setminus \cD$. If we put 
$\alpha' - \alpha = d \beta$ in $\Omega^i_{\cD/W_n}$, 
\begin{equation}\label{eq:compclaim1}
\alpha \wedge \dlog t - \alpha' \wedge \dlog t = 
d\beta \wedge \dlog t = d(\beta \wedge \dlog t)
\end{equation}
in $\Omega^{i+1}_{\cU''/W_n}/\Omega^{i+1}_{\cP_D/W_n}$. 
Also, if we put 
$ f = \sum_{n=1}^{\infty} (n-1)!(t/t' - 1)^{[n]} \in 
\cO_{\cU''}$, $df = \dlog t - \dlog t'$ in 
$\Omega^1_{\cU''/W_n}$, and so 
\begin{equation}\label{eq:compclaim2}
\alpha \wedge \dlog t - \alpha \wedge \dlog t' 
= \alpha \wedge df = d(\alpha \wedge f)
\end{equation}
(for the last equality, we used the equality $d\alpha=0$). 
By combining \eqref{eq:compclaim1}, \eqref{eq:compclaim2}, 
we get the equality 
$[\alpha \wedge \dlog t] = [\alpha' \wedge \dlog t']$ 
as required. 
\end{proof}

Note that the first isomorphism in \eqref{eq:ir} for different 
lifts $\cX, \cX'$ of $X$ is related by the commutative diagram
\[ 
\xymatrix{
\cH^i(\Omega^{\bullet}_{\cX/W_n}) 
\ar[d]^{\cong} \\
\cH^i(\Omega^{\bullet}_{\cD/W_n}) & 
Ru_{X.*}\cO_{X/W_n}, \ar[lu]^{\cong} \ar[l]^{\cong} \ar[ld]^{\cong} \\
\cH^i(\Omega^{\bullet}_{\cX'/W_n}) \ar[u]^{\cong}
}\]
where $\cD$ is as in the claim. The same property holds also for 
the lifts $\cP, \cP'$ of $P$. 
So, by the above claim, 
when we would like to prove the required compatibility, 
we may replace the morphism 
$\cX \hookrightarrow \cP$ by another regular closed immersion 
which lifts the morphism $X \hookrightarrow P$. Then, since we may work 
locally, we may assume that there exists a lift of Frobenius 
$\varphi: \cP \lra \cP$ on $\cP$ with $\varphi(t)=t^p$. 
In this case, the isomorphism \eqref{eq:irp} is the composite 
\begin{equation}\label{eq:pppp}
\cH^i(\Omega^{\bullet}_{\cP/W_n}) \os{\cong}{\lra} 
\cH^i(W_n\Omega^{\bullet}_{P})
\os{\cong}{\lla} W_n\Omega^i_P, 
\end{equation}
where the first isomorphism is induced by 
$\cO_{\cP} \lra W_n\cO_P$ satisfying $t \mapsto \ul{\ol{t}}$ with 
$\ul{\ol{t}}$ the Teichm\"uller lift of the mod 
$p$ reduction $\ol{t}$ of $t$, and the second isomorphism is 
the higher Cartier inverse isomorphism of Illusie-Raynaud \cite[III Prop.~1.4]{IR83}. Then we see that the isomorphism \eqref{eq:pppp}, with $i=1$ and 
$P,\cP$ replaced by $U = P \setminus X, \cU = \cP \setminus \cX$, 
sends $[\dlog t]$ to $[\dlog \ul{\ol{t}}]$. Hence, by the description of 
the maps $\cH^i(\psi_1^{\bullet})$, $\psi_{0,1}^i$ given in (1), (1)' above, 
we conclude that they are compatible, as required. 
\end{proof}


\begin{thebibliography}{99}




\bibitem{BS19}
F.~Binda and S.~Saito, 
{\it Relative cycles with moduli and regulator maps}, 
J. Inst. Math. Jussieu {\bf 18}(2019), 1233--1293. 

\bibitem{CR11} 
A.~Chatzistamatiou and K.~R\"ulling, 
{\it Higher direct images of the structure sheaf in positive characteristic}, 
Algebra and Number Theory {\bf 5}(2011), No.~6, 693--775. 

\bibitem{CR12}
A.~Chatzistamatiou and K.~R\"ulling, 
{\it Hodge-Witt cohomology and Witt-rational singularities}, 
Documenta Math. {\bf 17}(2012), 663--781. 







\bibitem{C00}
B.~Conrad, {\it Grothendieck duality and base change}, 
Lecture Notes in Math. {\bf 1750}, Springer, Berlin, 2000.

\bibitem{C11}
B.~Conrad, {\it Clarifications and corrections for 
Grothendieck duality and base change}. 



\bibitem{G85}
M.~Gros, 
{\it Class de Chern et classes de cycles en cohomologie de 
Hodge-Witt logarithmique}, 
M\'emoires de la S.~M.~F.~ {\bf 21}(1985), 1--87. 

\bibitem{H66}
R.~Hartshorne, 
{\it Residues and duality}, 
Lecture Notes in Math. {\bf 20}, Springer, Berlin, 1966. 

\bibitem{HK94}
O.~Hyodo and K.~Kato, 
{\it Semistable reduction and crystalline cohmology with logarithmic poles}, 
P\'eriodes $p$-adiques (Bures-sur-Yvette 1988), Ast\'erisque 
{\bf 223}(1994), 221--268. 

\bibitem{I79}
L.~Illusie, 
{\it Complexe de de Rham-Witt et cohomologie cristalline}, 
Ann. Sci. E.N.S. {\bf 12}(1979), No.~4, 501--661. 

\bibitem{IR83}
L.~Illusie and M.~Raynaud, 
{\it Les suites spectrales associ\'ees au complexe de de Rham-Witt}, 
Pub. Math. I.H.E.S. {\bf 57}(1983), 73--212. 

\bibitem{IY22}
F.~Ivorra and T.~Yamazaki, 
{\it Mixed Hodge structure with modulus}, 
J. Inst. Math. Jussieu {\bf 21}(2022), 161--195. 

\bibitem{KMSY21a}
B.~Kahn, H.~Miyazaki, S.~Saito and T.~Yamazaki, 
{\it Motives with modulus I: Modulus sheaves with transfers for non-proper modulus pairs}, \'Epijournal G\'eom. Alg\'ebrique {\bf 5}(2021). 

\bibitem{KMSY21b}
B.~Kahn, H.~Miyazaki, S.~Saito and T.~Yamazaki, 
{\it Motives with modulus II: Modulus sheaves with transfers for proper modulus pairs}, \'Epijournal G\'eom. Alg\'ebrique {\bf 5}(2021). 

\bibitem{KMSY22} 
B.~Kahn, H.~Miyazaki, S.~Saito and T.~Yamazaki, 
{\it Motives with modulus III: The category of motives}, 
Annals of $K$-theory {\bf 7.1}(2022), 119--178. 



%
%
%

\bibitem{KSY16}
B.~Kahn, S.~Saito and T.~Yamazaki, 
{\it Reciprocity sheaves (with two appendices by Kay R\"ulling)}, 
Compositio Math. {\bf 152}(2016), No.~9, 1851--1898. 

\bibitem{KM23a}
S.~Kelly and H.~Miyazaki, 
{\it Hodge cohomology with a ramification filtration I}, 
Math. Z. {\bf 305} 305:70. 

\bibitem{KM23b}
S.~Kelly and H.~Miyazaki, 
{\it Hodge cohomology with a ramification filtration II}, 
arXiv:2306.06864v1

\bibitem{K23}
J.~Koizumi, 
{\it Blow-up invariance of cohomology theories with modulus}, 
arXiv:2306.14803v2


\bibitem{M93}
A.~Mokrane, 
{\it La suite spectrale des poids en cohomologie de Hyodo-Kato}, 
Duke Math. J. {\bf 72}(1993), 301--337. 

\bibitem{N05}
Y.~Nakkajima, 
{\it $p$-adic weight spectral sequences of log varieties}, 
J.~Math.~Sci.~Univ.~Tokyo {\bf 12}(2005), 513--661. 

\bibitem{NS08} 
Y.~Nakkajima and A.~Shiho, 
{\it Weight filtrations and log crystalline cohomologies of families of open smooth varieties}, Lecture Note in Math. {\bf 1959}, Springer Verlag. 

\bibitem{T99}
T.~Tsuji, 
{\it Poincar\'e duality for logarithmic crystalline cohomology}, 
Compositio Math.~{\bf 118}(1999), 11--41. 

%
%
%
%
%
%
%
%



\end{thebibliography}
\end{document}